\documentclass{amsart}

\usepackage[utf8]{inputenc} % for utf8 encoding 
\usepackage[T1]{fontenc} % font enconding is set to T1
\usepackage{amsmath,amssymb,amsthm,amsfonts} % ams packages
\usepackage{hyperref} % for hyperlinks
\usepackage{tikz-cd} % commutative diagrams
\usepackage{tikz} % for drawing figures
\usetikzlibrary{arrows,positioning} 

\newcommand{\nn}{\mathbb{N}} % natural numbers
\newcommand{\isom}{\cong} % isomorphisms
\newcommand{\ext}{\mathbf{E}} % extension graphs
\newcommand{\rext}{\mathbf{R}} % right extensions
\newcommand{\lext}{\mathbf{L}} % left extensions
\newcommand{\edge}{\mathcal{E}} % graph edges
\newcommand{\vertex}{\mathcal{V}} % graph vertices
\newcommand{\ret}{\mathcal{R}} % return sets 
\newcommand{\emptyw}{\varepsilon} % empty word

\DeclareMathOperator{\Card}{Card} % set cardinality
\DeclareMathOperator{\img}{Im} % image of map
\DeclareMathOperator{\eval}{eval} % evaluation map
\DeclareMathOperator{\tail}{tail} % tail map
\DeclareMathOperator{\init}{init} % init map
\DeclareMathOperator{\rank}{rank} % rank (of free groups)

\theoremstyle{plain}
\newtheorem{theorem}{Theorem}[section]
\newtheorem{lemma}[theorem]{Lemma}
\newtheorem{proposition}[theorem]{Proposition}
\newtheorem{corollary}[theorem]{Corollary}

\theoremstyle{definition} 
\newtheorem{definition}[theorem]{Definition}

\theoremstyle{remark}
\newtheorem{remark}[theorem]{Remark}
\newtheorem{example}[theorem]{Example}
\newtheorem{question}[theorem]{Question}
\newtheorem*{claim}{Claim}

\begin{document}

\title{Suffix-connected languages}  

\thanks{%
    The author is grateful for the financial support provided by the Centre for Mathematics of the University of Coimbra (UIDB/00324/2020, funded by the Portuguese Government through FCT/MCTES), the Centre for Mathematics of the University of Porto (UIDB/00144/2020, funded by the Portuguese Government through FCT/MCTES), as well as a PhD grant from FCT/MCTES (PD/BD/150350/2019). Special thanks go to Jorge Almeida and Alfredo Costa for many helpful discussions and comments which greatly improved this paper.}

\author[H.~Goulet-Ouellet]{Herman Goulet-Ouellet}

\address{University of Coimbra, CMUC, Department of Mathematics,
    and University of Porto, CMUP, Department of Mathematics}

\email{hgouletouellet@student.uc.pt}



\begin{abstract}
    Inspired by a series of papers initiated in 2015 by Berthé et al., we introduce a new condition called suffix-connectedness. We show that the groups generated by the return sets of a uniformly recurrent suffix-connected language lie in a single conjugacy class of subgroups of the free group. Moreover, the rank of the subgroups in this conjugacy class only depends on the number of connected components in the extension graph of the empty word. We also show how to explicitly compute a representative of this conjugacy class using the first order Rauzy graph. Finally, we provide an example of suffix-connected, uniformly recurrent language that contains infinitely many disconnected words. 
\end{abstract}

\maketitle

\section{Introduction}
\label{s:intro}

In \cite{Berthe2015}, Berthé et al.\ introduced the notion of extension graph and used it to study the subgroups generated by the return sets in uniformly recurrent languages. One result achieved in that paper, dubbed the \emph{Return Theorem}, states that if $L$ is a uniformly recurrent language on the alphabet $A$ such that all the extension graphs of $L$ are trees, then the return sets of $L$ are all bases of the free group on $A$ \cite[Theorem~4.5]{Berthe2015}. Moreover, they also show that part of this result holds under weaker assumptions: if we merely assume that the extension graphs of $L$ are connected, then the return sets of $L$ all generate the free group on $A$ \cite[Theorem~4.7]{Berthe2015}. The aim of this paper is to give a weaker condition under which a similar conclusion still holds. To do this, we introduce \emph{suffix extension graphs}, a notion generalizing the extension graphs of \cite{Berthe2015}. This allows us to define a new condition called \emph{suffix-connectedness}. Our main result is the following:
\begin{theorem}\label{t:main}
    Let $L$ be a suffix-connected uniformly recurrent language on an alphabet $A$. Then the subgroups generated by the return sets of $L$ all lie in the same conjugacy class and their rank is $n-c+1$, where $n=\Card(A)$ and $c$ is the number of connected components of the extension graph of the empty word.
\end{theorem}

Our proof is constructive, in the sense that we can also deduce a way to explicitly compute a representative for this conjugacy class. Moreover, the proof of Theorem~\ref{t:main} has two notable consequences that we wish to highlight now. The first one is a characterization of suffix-connected, uniformly recurrent languages whose return sets generate the full free group.
\begin{corollary}\label{c:main1}
    Let $L$ be a suffix-connected and uniformly recurrent language on the alphabet $A$. Then the following statements are equivalent: 
    \begin{enumerate}
        \item All the return sets of $L$ generate the free group on $A$.
        \item Some return set of $L$ generates a group of rank $\Card(A)$.
        \item The extension graph of the empty word is connected.
    \end{enumerate}
\end{corollary}

The next corollary is a special case of our main result. It involves \emph{neutrality}, which is a combinatorial condition also introduced in \cite{Berthe2015} (we will recall the definition in Section~\ref{s:cor}). A \emph{connected set} is a language in which the extension graphs of non-empty words are connected, while a \emph{tree set} is a language in which the extension graph of the empty word is a forest, and all other extension graphs are trees. These conventions differ slightly from \cite{Berthe2015}, but are in line with other papers such as \cite{Berthe2017a,Dolce2017}. The term \emph{dendric} has also been used to refer to tree sets, for instance in \cite{Dolce2020}. A subset of the free group is called \emph{free} if it forms a basis of the subgroup it generates.
\begin{corollary}\label{c:main2}
    Let $L$ be a uniformly recurrent language on the alphabet $A$. If $L$ is connected and neutral, then the following conditions are equivalent:
    \begin{enumerate}
        \item Some return set of $L$ is a free subset of the free group on $A$.
        \item All return sets of $L$ are free subsets of the free group on $A$.
        \item $L$ is a tree set.
    \end{enumerate}
\end{corollary}
Since connectedness implies suffix-connectedness, the assumptions of the Return Theorem place us in the scope of both Corollary~\ref{c:main1} and \ref{c:main2}. It follows that the Return Theorem is a direct consequence of the above corollaries.

In order to further motivate this new \emph{suffix-connectedness} condition, we give an example of a uniformly recurrent language which is suffix-connected but contains infinitely many disconnected elements. This language is defined by a primitive substitution. More precisely, we will show the following:
\begin{theorem}\label{t:example}
    The language of the primitive substitution
    \begin{equation*}
        \begin{array}{llll}
            \varphi\colon & 0 & \mapsto & 0001 \\
                     & 1 & \mapsto & 02 \\
                     & 2 & \mapsto & 001 
        \end{array}
    \end{equation*}
    is suffix-connected and contains infinitely many disconnected words.
\end{theorem}
We will also see that, in the language of this substitution, the extension graph $\ext(\emptyw)$ is connected. Therefore, as a result of Corollary~\ref{c:main1}, all the return sets in this language generate the full free group of rank 3. However, further computations reveal that the language of $\varphi$ has return sets of cardinality 3 and 4, which means that some but not all of them are free subsets of the free group. 

This paper is structured as follows. In Section~\ref{s:connectedness}, we introduce suffix extension graphs and suffix-connectedness, while also recalling some relevant definitions in more details. Section~\ref{s:stallings} reviews some basic material about the groups generated by labeled digraphs. Section~\ref{s:rauzy} is devoted to Rauzy graphs. Section~\ref{s:paths} presents a technical result that makes up the core of the proof of our main result. In Section~\ref{s:return}, we examine the relationship between Rauzy graphs and return sets. In Section~\ref{s:main}, we put everything together and give the proof of Theorem~\ref{t:main}. Section~\ref{s:cor} discusses the proof of the two corollaries above. Finally, Section~\ref{s:example} is devoted to our suffix-connected example. 

\section{Suffix-connectedness}
\label{s:connectedness}

In this paper, $L$ denotes a language on a finite alphabet $A$ of cardinality $n$, and $F(A)$ denotes the free group on $A$. We will always suppose that $L$ is recurrent and that $A \subseteq L$. We recall that a language $L$ is \emph{recurrent} if it is closed under taking factors, and if for every two words $u, v\in L$, there exists a non-empty word $w$ such that $uwv\in L$. The \emph{left extensions} and \emph{right extensions} of order $k$ of $w\in L$, are:
\begin{equation*}
    \lext_k(w) = \{ u\in L\cap A^k : uw\in L\}, \qquad \rext_k(w) = \{ v\in L\cap A^k : wv\in L\}.
\end{equation*}

The \emph{extension graph of order $(k,l)$} of $w\in L$ is a bipartite graph over the disjoint union of $\lext_k(w)$ and $\rext_l(w)$ (the union of disjoint copies of $\lext_k(w)$ and $\rext_l(w)$). In this graph, there is an edge between $u\in\lext_k(w)$ and $v\in\rext_l(w)$ if $uwv\in L$. We denote this graph by $\ext_{k,l}(w)$. Note that all extension graphs are simple and undirected. We abbreviate $\rext_1$, $\lext_1$ and $\ext_{1,1}$ respectively by $\rext$, $\lext$ and $\ext$. In the absence of further clarifications, the term \emph{extension graph of $w$} refers to $\ext(w)$. A word is \emph{connected} if its extension graph is connected, and it is called \emph{disconnected} otherwise. A language is \emph{connected} if all its non-empty words are connected, and it is \emph{disconnected} otherwise.

A word $w\in L$ is called \emph{left special} if $\Card(\lext(w))>1$. Similarly, $w$ is called \emph{right special} if $\Card(\rext(w))>1$. By a \emph{bispecial word}, we mean a word which is both left and right special. 

\begin{remark}
    If $w$ is not bispecial, then $\ext(w)$ is a star graph, and in particular a tree. Hence, only bispecial factors can be disconnected.
\end{remark}

Given a word $w\in A^*$ and $0\leq i<|w|$, we denote by $w(i)$ the $i$-th letter of $w$. In particular, the first letter of $w$ is $w(0)$. Given $0\leq i\leq j\leq |w|$, we denote $w[i:j]$ the factor of $w$ defined by:
\begin{equation*}
    w[i:j] = w(i)w(i+1)\dots w(j-1)
\end{equation*}
Note that $|w[i:j]|=j-i$, $w[i:i]$ is the empty word and $w[0:|w|]=w$. Let $u\in A^*$ with $|u|=k$. We say that an index $j$ is an \emph{occurrence} of $u$ in $w$ if $w[j:j+k] = u$. We also define the \emph{tail} and the \emph{init} of a non-empty word $w$ by putting:
\begin{equation*}
    \tail(w) = w[1:|w|],\quad \init(w) = w[0:|w|-1].
\end{equation*}
We view $\tail$ and $\init$ as maps $A^+\to A^*$. With this, we are now ready to introduce suffix extension graphs.

\begin{definition}
    For $w\in L$ and $1\leq d\leq |w|+1$, the \emph{depth $d$ suffix extension graph} of $w$ is the extension graph $\ext_{d,d}(\tail^{d-1}(w))$. 
\end{definition}

The set $\lext(w)$ naturally embeds in the suffix-extension graphs of $w$. Indeed, let $u$ be the prefix of length $d-1$ of $w$, which means that $u$ satisfies $w=u\tail^{d-1}(w)$. Then $a\mapsto au$ is an injective map $\lext(w)\to \lext_d(\tail^{d-1}(w))$, with the latter set being viewed as a subset of $\ext_{d,d}(\tail^{d-1}(w))$. We call this the \emph{natural embedding} of $\lext(w)$ in the depth $d$ suffix extension graph.

\begin{definition}\label{d:suffix-connected}
    A word $w$ is called \emph{suffix-connected} if the natural embedding of $\lext(w)$ in $\ext_{d,d}(\tail^{d-1}(w))$ lies in one connected component, for some $1\leq d\leq |w|+1$. A language is called suffix-connected if all its non-empty words are suffix-connected.
\end{definition}

We note that this definition is sensitive to both increases and decreases in the depth parameter. That is, for a given word $w$, it may happen that some of the natural embeddings $\lext(w)$ lie in a single connected component, while others do not. The next example is a good illustration of this behaviour. It features a language defined by a primitive substitution, and such languages are well known to be uniformly recurrent (see for instance \cite[Proposition~1.2.3]{Fogg2002}). 
\begin{example}
    Let us consider the following binary substitution, known as the \emph{Thue-Morse} substitution:
    \begin{equation*}
        \begin{array}{rrrr}
            \mu\colon & 0 & \mapsto & 01\\
                 & 1 & \mapsto & 10
        \end{array}.
    \end{equation*}

    Let $L$ be the language defined by $\mu$. That is, $L$ is the set of factors of all words of the form $\mu^n(a)$ for $n\in\nn$ and $a\in\{0,1\}$. Figure~\ref{f:suff-ext-morse-1} gives all the suffix extension graphs of the word $010\in L$, which show that $L$ is not suffix-connected. 

    On the other hand, Figure~\ref{f:suff-ext-morse-2}, gives some extension graphs of $01100\in L$. These graphs show that the natural embeddings of a given word can alternate between being connected and disconnected as the depth increases.
\end{example}

\begin{figure}\centering
    \begin{minipage}[b]{.23\linewidth}\centering
        \includegraphics[scale=.4]{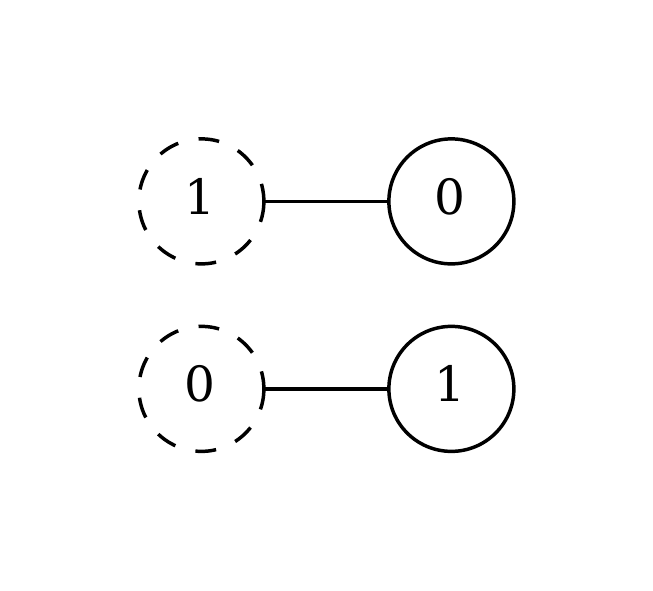}\\
        {\small$\ext_{1,1}(010)$}
    \end{minipage}
    \begin{minipage}[b]{.23\linewidth}\centering
        \includegraphics[scale=.3]{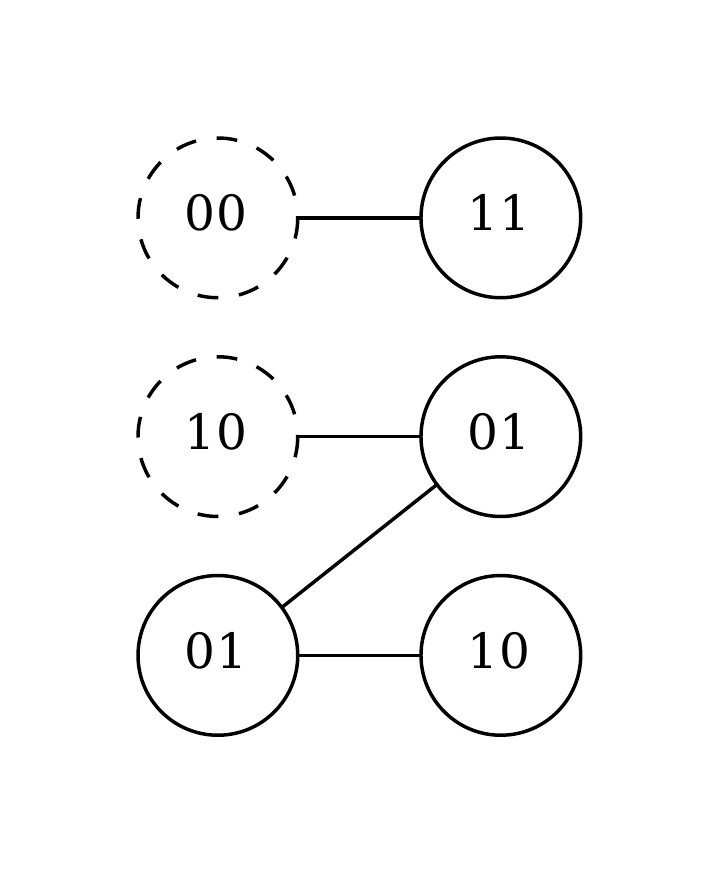}\\
        {\small$\ext_{2,2}(10)$}
    \end{minipage}
    \begin{minipage}[b]{.23\linewidth}\centering
        \includegraphics[scale=.25]{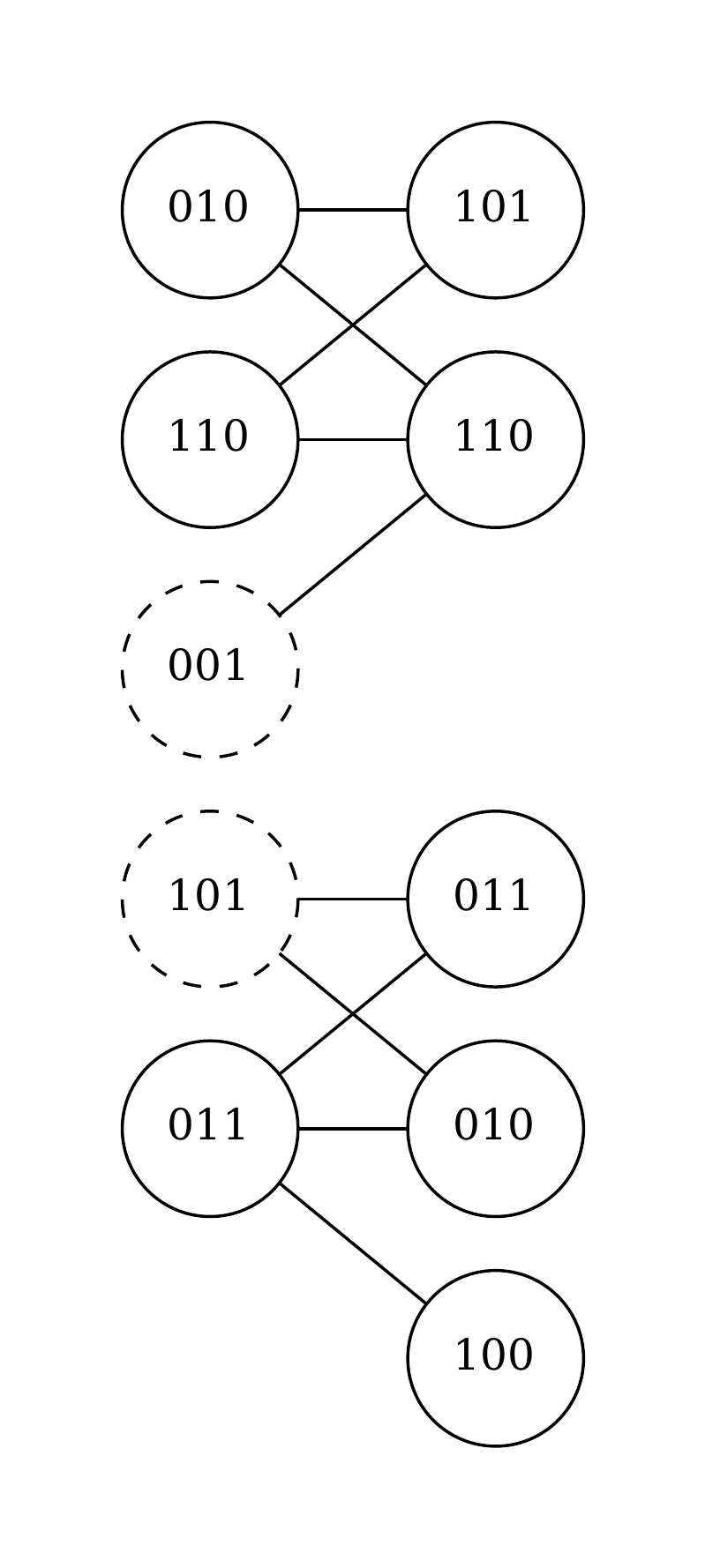}\\
        {\small$\ext_{3,3}(0)$}
    \end{minipage}
    \begin{minipage}[b]{.23\linewidth}\centering
        \includegraphics[scale=.2]{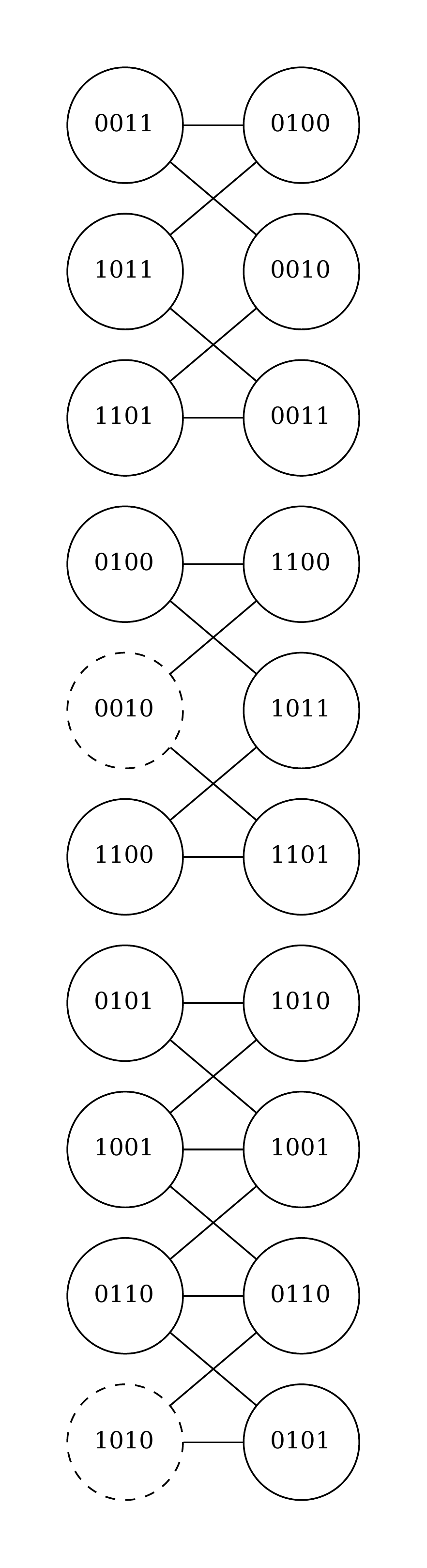}\\
        {\small$\ext_{4,4}(\emptyw)$}
    \end{minipage}
    \caption{Suffix extension graphs of the word $010$ in the language of the Thue-Morse substitution. The dashed vertices represent the natural embeddings of $\lext(010)$.}\label{f:suff-ext-morse-1}
\end{figure}

\begin{figure}\centering
    \begin{minipage}[b]{.23\linewidth}\centering
        \includegraphics[scale=.4]{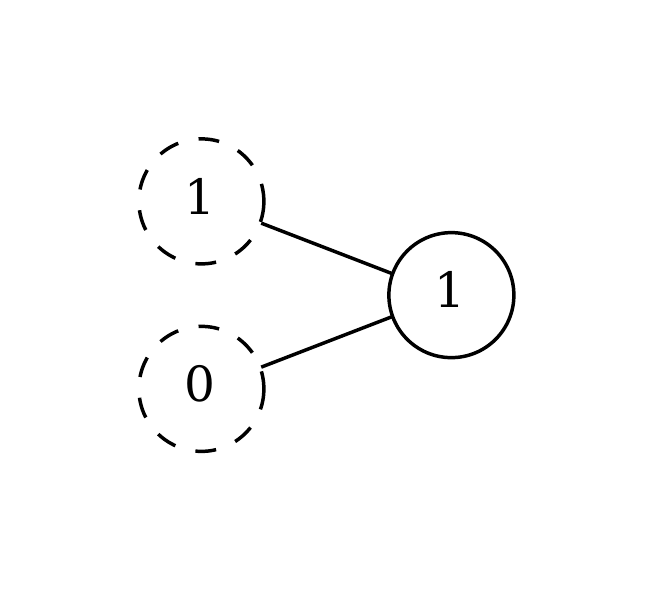}\\
        {\small$\ext_{1,1}(01100)$}
    \end{minipage}
    \begin{minipage}[b]{.23\linewidth}\centering
        \includegraphics[scale=.35]{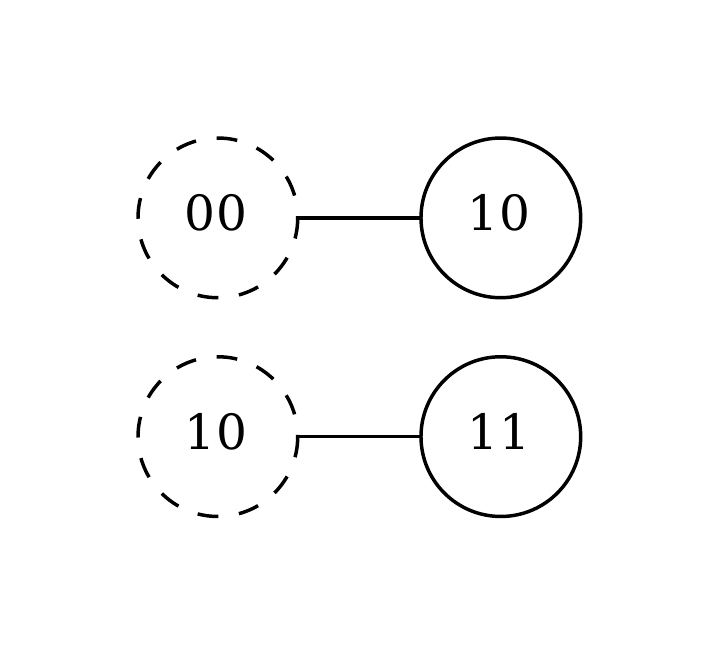}\\
        {\small$\ext_{2,2}(1100)$}
    \end{minipage}
    \begin{minipage}[b]{.23\linewidth}\centering
        \includegraphics[scale=.3]{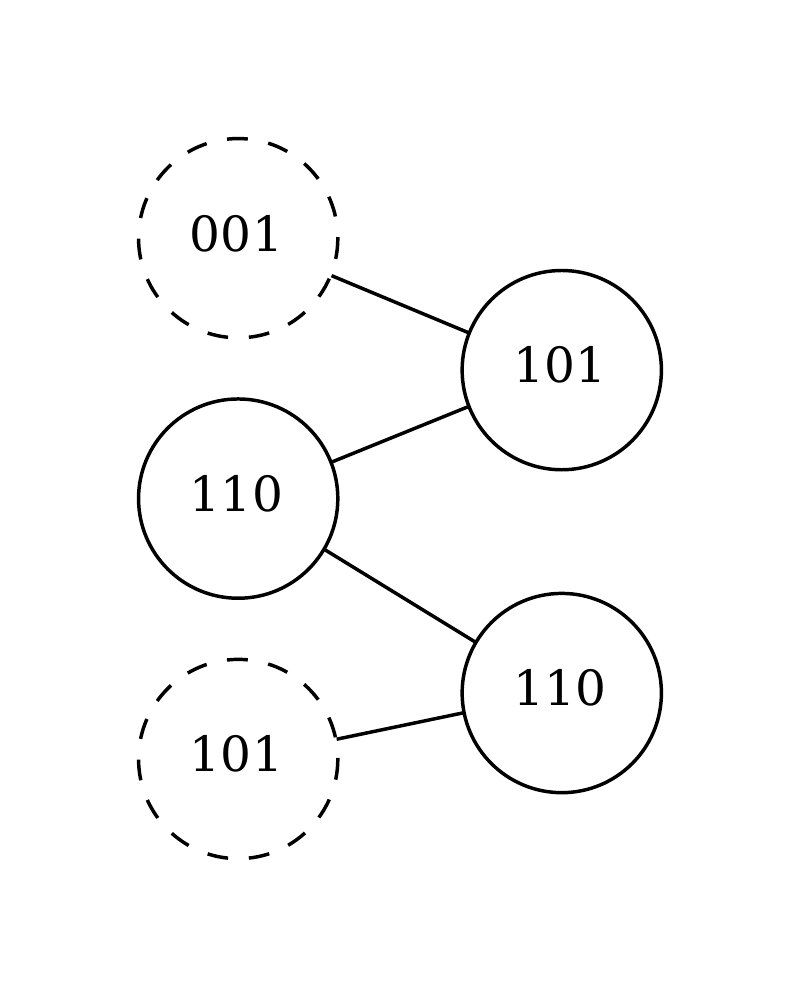}\\
        {\small$\ext_{3,3}(100)$}
    \end{minipage}
    \begin{minipage}[b]{.23\linewidth}\centering
        \includegraphics[scale=.25]{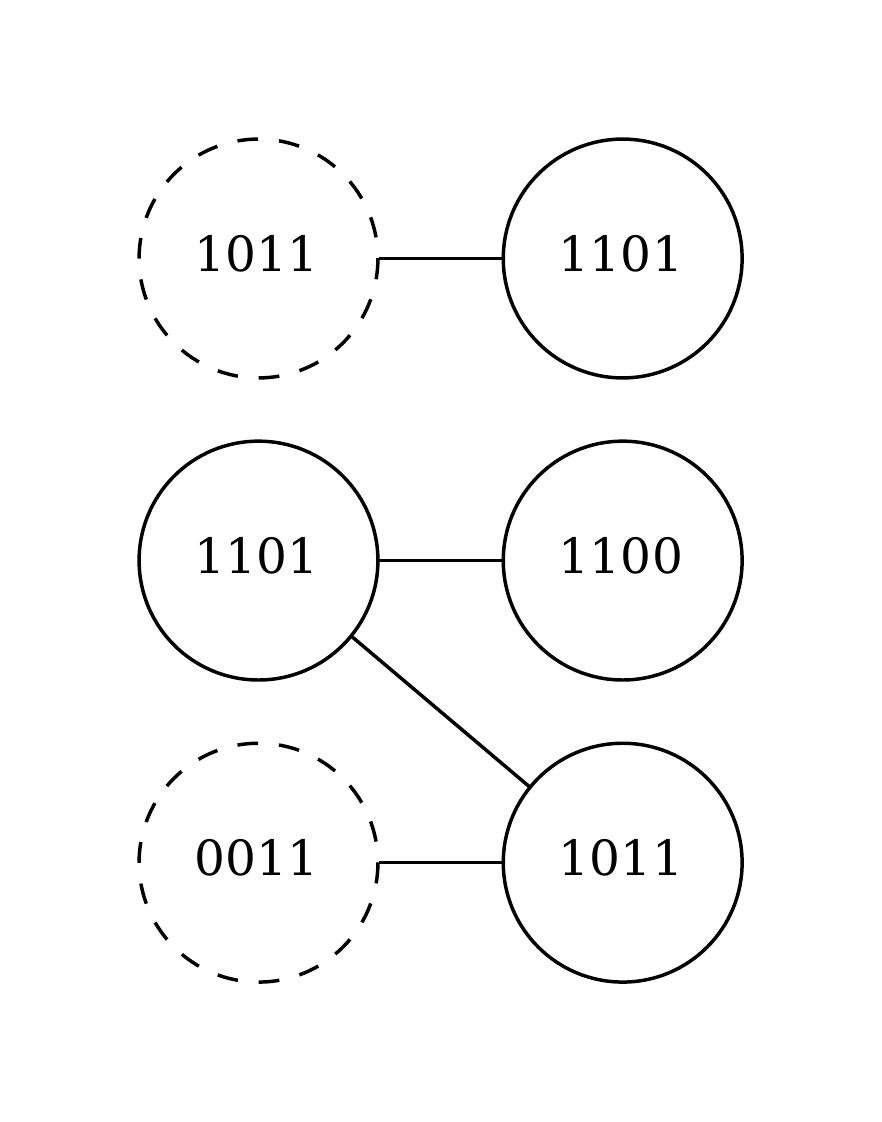}\\
        {\small$\ext_{4,4}(00)$}
    \end{minipage}
    \caption{First four suffix extension graphs of the word $01100$ in the language of the Thue-Morse substitution. The dashed vertices represent the natural embeddings of $\lext(01100)$.}\label{f:suff-ext-morse-2}
\end{figure}

Replacing $\tail$ by $\init$ and $\lext$ by $\rext$ yields the dual notions of \emph{prefix extension graphs} and \emph{prefix-connectedness}. Note that the depth~1 suffix and prefix extension graphs of $w$ both coincide with $\ext(w)$, so a connected word or language is both prefix and suffix-connected.

\section{Stallings equivalence}
\label{s:stallings}

Let us start this section by clarifying some basic terminology. A \emph{labeled digraph} over the alphabet $A$ (or, more simply, a \emph{digraph}) is a diagram of sets $G$ of the following form:
\begin{center}
    \begin{tikzcd}
       &\edge(G) \arrow[dr,shift left,"\alpha"] \arrow[dr,shift right,"\omega",swap] \arrow[dl,"\lambda",swap]& \\
        A& &\vertex(G)
    \end{tikzcd}
\end{center}
One can think of $\vertex$ as the set of \emph{vertices}, $\edge$ as the set of \emph{edges}, and $A$, the alphabet, as the set of \emph{labels}. The maps $\alpha$, $\omega$ and $\lambda$ give us respectively the origin, terminus and label of a given edge. For our purposes, we may assume that there are no redundant edges, meaning that $(\alpha,\lambda,\omega)$ are jointly injective. This means in effect that $\edge$ may be considered a subset of $\vertex\times A\times\vertex$ whenever convenient.

Given an edge $e = (x,a,y)$, we consider its formal inverse $e^{-1} = (y,a^{-1},x)$. From now on, we use the term \emph{edge} both for elements of $\edge(G)$ and for their formal inverses. Two edges are said to be \emph{consecutive} if the last component of the first is equal to the first component of the second. A \emph{path} is a sequence of consecutive edges. We can naturally extend the maps $\alpha,\omega$ to paths, and talk about \emph{consecutive paths}. Two consecutive paths can be composed, and any path can be inverted; we write respectively $pq$ and $p^{-1}$. A self-consecutive path is called a \emph{loop}. As expected, if $p,q$ are consecutive, then so are $q^{-1},p^{-1}$ and the relation $(pq)^{-1} = q^{-1}p^{-1}$ holds.

The labeling map $\lambda$ also naturally extends, mapping the set of all paths to the free group $F(A)$. This map satisfies $\lambda(pq)=\lambda(p)\lambda(q)$ and $\lambda(p^{-1})=\lambda(p)^{-1}$. We write $p\colon x\stackrel{u}{\to}y$ as a shorthand for $\alpha(p)=x$, $\omega(p)=y$, $\lambda(p)=u$. The set of all labels of loops over a given vertex $x$ forms a subgroup of $F(A)$, which we call the \emph{group of $G$ at $x$}. Note that under the assumption that $G$ is connected (any two vertices can be joined by a path), all the groups of $G$ lie in the same conjugacy class of subgroups of $F(A)$.

Let $\equiv$ be an equivalence relation on the vertices of a digraph $G$. Then $\equiv$ can also be seen as an equivalence relation on $\edge(G)$,
\begin{equation*}
    (x,a,y)\equiv(x',b,y') \iff x\equiv x', a=b, y\equiv y'.
\end{equation*}
The \emph{quotient digraph} $G/{\equiv}$ is then defined by:
\begin{equation*}
    \vertex(G/{\equiv})=\vertex(G)/{\equiv},\quad \edge(G/{\equiv})=\edge(G)/{\equiv},
\end{equation*}
together with the following adjancency and labeling maps:
\begin{equation*}
    \alpha(x/{\equiv}) = \alpha(x)/{\equiv},\quad \omega(x/{\equiv}) =  \omega(x)/{\equiv},\quad \lambda(x/{\equiv}) = \lambda(x).
\end{equation*}
The definition of $G/{\equiv}$ can be summarized by the following commutative diagrams:
\begin{center}
    \begin{tikzcd}[row sep=large]
        A & \lar[swap]{\lambda} \dar[swap]{\equiv} \edge(G) \\
          & \ular{\lambda} \edge(G/{\equiv})
    \end{tikzcd}
    \quad
    \begin{tikzcd}[row sep=large]
        \edge(G) \rar{\alpha} \dar[swap]{\equiv} & \vertex(G) \dar[swap]{\equiv} \\
        \edge(G/{\equiv}) \rar{\alpha} & \vertex(G/{\equiv})
    \end{tikzcd}
    \quad
    \begin{tikzcd}[row sep=large]
        \edge(G) \rar{\omega} \dar[swap]{\equiv} & \vertex(G) \arrow[swap,d,"\equiv"] \\
        \edge(G/{\equiv}) \rar{\omega} & \vertex(G/{\equiv})
    \end{tikzcd}
\end{center}
The natural projection $G\to G/{\equiv}$ is a \emph{digraph morphism}, meaning that it preserves the maps $\alpha,\omega,\lambda$. If, conversely, $\phi\colon G\to H$ is a digraph morphism, then the quotient $G/\ker(\phi)$ is isomorphic to $\img(\phi)$, where $\ker(\phi) = \{ (x,y) : \phi(x)=\phi(y)\}$. Note that for a digraph morphism $\phi\colon G\to G'$ to be onto, it needs to be onto on both $\vertex(G')$ and $\edge(G')$. The latter condition can be written as follows:
\begin{equation*}
    \forall (x,a,y)\in \edge(G'), \exists (x',a,y')\in\edge(G), \ \phi(x')=x \land \phi(y')=y.
\end{equation*}

We say that an equivalence relation $\equiv$ on $\vertex(G)$ is \emph{group-preserving} if the group of $G$ at $x$ is equal to the group of $G/{\equiv}$ at $x/{\equiv}$, for all $x\in\vertex(G)$. We also call \emph{group-preserving} a digraph morphism whose kernel is a group-preserving relation. Note that the group of $G$ at $x$ is always a subgroup of the group of $G/{\equiv}$ at $x/{\equiv}$. Therefore, to prove that $\equiv$ is group-preserving, one only needs to prove the reverse inclusion. Moreover, in the case of a connected digraph, this inclusion needs only to be checked on a single vertex. 

The family of group-preserving equivalence relations of a digraph $G$ also has the property of being closed under taking subrelations. Indeed, let us suppose that $\equiv_1$ is group-preserving and consider ${\equiv_2}\subseteq{\equiv_1}$. Then, the canonical surjection of $\equiv_1$ factors through that of $\equiv_2$, giving us the following commutative diagram:
\begin{center}
    \begin{tikzcd}
        G \dar[swap]{\equiv_2} \rar{\equiv_1} & G/{\equiv_1} \\
        G/{\equiv_2} \urar[swap,dashed] & 
    \end{tikzcd}
\end{center}
Let us fix $x\in\vertex(G)$ and let $H$, $H_1$, $H_2$ be respectively the group of $G$ at $x$; the group of $G/{\equiv_1}$ at $x/{\equiv_1}$; and the group of $G/{\equiv_2}$ at $x/{\equiv_2}$. Then the diagram above implies $H\leq H_2\leq H_1$, while the fact that $\equiv_1$ is group-preserving implies $H = H_1$. Thus, $H_2=H$ and $\equiv_2$ is also group-preserving.

A well-known algorithm due to Stallings implies that a digraph always has a greatest group-preserving equivalence relation. We now proceed to give a description of this equivalence relation, starting with the following definition. 
\begin{definition}
    The \emph{Stallings equivalence} of $G$ is the least equivalence relation on $\vertex(G)$ closed under the two following rules:
    \begin{itemize}
        \item[(F)] If $(x,y)$, $(u,x')$, $(y',v)$ are related, and $(x,a,x')$, $(y,a,y')$ are edges in $G$, then $(u,v)$ are related.
        \item[(F\textquotesingle)] If $(x,y)$, $(u,x')$, $(y',v)$ are related, and $(x',b,x)$, $(y',b,y)$ are edges in $G$, then $(u,v)$ are related.
    \end{itemize}
    We denote the Stallings equivalence by $\equiv_S$.
\end{definition}

Note that if two equivalence relations are closed under either rule (F) or (F\textquotesingle), then so is their intersection (this follows immediately from the definitions). Moreover, the total relation $\vertex(G)\times\vertex(G)$ is trivially closed under the two rules. Hence, the relation $\equiv_S$ is simply the intersection of all equivalence relations on $\vertex(G)$ that are closed under (F) and (F\textquotesingle).

\begin{figure}\centering
    \begin{minipage}[b]{0.4\linewidth}\centering
        \begin{tikzpicture}[bend angle=20,auto]
            \node (x) at (0,0) {$x$};
            \node (y) at ([shift=({0:1.3 cm})]x) {$y$};
            \node (v) at ([shift=({90:1.3 cm})]y) {$v$};
            \node (u) at ([shift=({180:1.3 cm})]v) {$u$};
            \node (x') at ([shift=({135:1.3 cm})]u) {$x'$};
            \node (y') at ([shift=({45:1.3 cm})]v) {$y'$};
            \draw (x) -- (y); 
            \draw[ultra thick] (x) -- (y); 
            \draw[ultra thick] (x') -- (u); 
            \draw[ultra thick] (y') -- (v); 
            \draw[dashed, thick] (v) -- (u); 
            \draw[->,>=latex,bend left] (x) to node {$a$} (x') ;
            \draw[->,>=latex,bend right] (y) to node[swap] {$a$} (y') ;
        \end{tikzpicture}\\
        {\small(F)}
    \end{minipage}
    \begin{minipage}[b]{0.4\linewidth}\centering
        \begin{tikzpicture}[bend angle=20,auto]
            \node (x) at (0,0) {$x$};
            \node (y) at ([shift=({0:1.3 cm})]x) {$y$};
            \node (v) at ([shift=({90:1.3 cm})]y) {$v$};
            \node (u) at ([shift=({180:1.3 cm})]v) {$u$};
            \node (x') at ([shift=({135:1.3 cm})]u) {$x'$};
            \node (y') at ([shift=({45:1.3 cm})]v) {$y'$};
            \draw[ultra thick] (x) -- (y); 
            \draw[ultra thick] (x') -- (u); 
            \draw[ultra thick] (y') -- (v); 
            \draw[dashed, thick] (v) -- (u); 
            \draw[->,>=latex,bend right] (x') to node[swap] {$b$} (x) ;
            \draw[->,>=latex,bend left] (y') to node {$b$} (y) ;
        \end{tikzpicture}\\
        {\small(F\textquotesingle)}
    \end{minipage}
    \caption{The rules defining Stallings equivalence. The arrows represent edges, the thick lines represent existing relations, and the dashed lines represent the relations deduced from each rule.}
\end{figure}
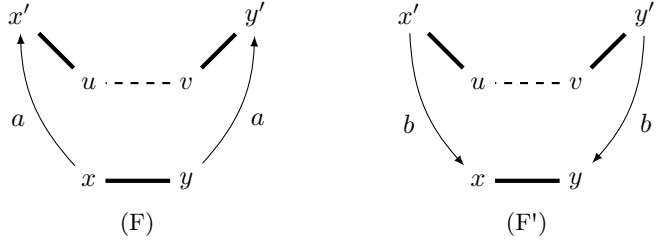

By a \emph{trivially-labeled path}, we mean a path whose label is the identity element of $F(A)$. The next result relates Stallings equivalence with trivially-labeled paths, and can be seen as a reformulation of Stallings algorithm. 
\begin{proposition}\label{p:stallings}
    Let $G$ be a connected digraph. The equivalence $\equiv_S$ is, alternatively,
    \begin{enumerate}
        \item the equivalence relation induced by trivially-labeled paths;
        \item the greatest group-preserving equivalence of $G$.
    \end{enumerate}
\end{proposition}

For the proof of this result, the following definition will be useful: given an equivalence relation $\equiv$ on $\vertex(G)$, an \emph{$\equiv$-path} in $G$ is a sequence of edges $p = (e_1, \dots, e_k)$ satisfying $\alpha(e_{i+1})\equiv\omega(e_i)$. The notions of \emph{label} and \emph{length} extend in a straightforward way to $\equiv$-paths. We also use the notation $p\colon x\stackrel{u}{\to} y$ for $\equiv$-paths, to mean $\alpha(p)=x$, $\omega(p)=y$ and $\lambda(p)=u$. Finally, we adopt the convention that an $\equiv$-path of length 0 is a pair $x\equiv y$. 

\begin{proof}[Proof of Proposition~\ref{p:stallings}]
    \textbf{(1)} Let us denote by $\sim$ the relation induced by trivially-labeled paths and by $\approx$ the relation induced by trivially-labeled $\equiv_S$-paths. Clearly $\sim$ is contained in $\approx$. Let us show that $\approx$ is contained in  $\equiv_S$.

    We proceed by induction on the length of the trivially-labeled $\equiv_S$-path. Note that by definition, an $\equiv_S$-path of length 0 is nothing but a pair $x\equiv y$, so the basis of the induction is trivial. Let us suppose that there is a trivially-labeled $\equiv_S$-path $p\colon x\to y$ of length $k\geq 1$. Write $p = (e_1, \dots, e_k)$. Since $p$ is trivially-labeled, $k$ is even and there must exist $i$ such that $\lambda(e_i) = a^{-1}$ and $\lambda(e_{i+1})=a$, where $a$ is either a letter, or the inverse of a letter. Write $e_i = (u',a^{-1},u)$ and $e_{i+1}=(v,a,v')$, where $u\equiv_S v$. If $a\in A$, then we may use rule (F) to conclude $u'\equiv_S v'$. Otherwise, one uses rule (F\textquotesingle) to obtain the same conclusion. It follows that $p' = (e_1, \dots, e_{i-1}, e_{i+2}, \dots, e_k)$ is also a trivially-labeled $\equiv_S$-path between $x$ and $y$. Since $p'$ has length $k-2<k$, we conclude by induction that $x\equiv_Sy$. 

    We finish the proof of (1) by showing that $\equiv_S$ is contained in $\sim$. By definition of $\equiv_S$, it suffices to show that $\sim$ is closed under the rules (F) and (F\textquotesingle). Suppose that $u\sim x'$, $x\sim y$, $y'\sim v$, and that there are two edges $e=(x,a,x')$ and $f=(y,a,y')$. Consider trivially-labeled paths $p_1\colon u\to x'$, $q\colon x\to y$ and $p_2\colon y'\to v$. Then, the composition $p_1e^{-1}qfp_2$ is a trivially-labeled path in $G$ between $u$ and $v$. Thus, $u\sim v$, which proves $\sim$ is closed under (F). The proof for (F\textquotesingle) is similar.

    \textbf{(2)} We first show that $\equiv_S$ is a group-preserving equivalence relation, and then we show it is the greatest. Let us fix any path $p\colon x/{\equiv_S}\to y/{\equiv_S}$ in the quotient $G/{\equiv_S}$. We say that an $\equiv_S$-path $q$ in $G$ \emph{lifts $p$} if $q\colon x'\to y'$ with $x\equiv_S x'$, $y\equiv_S y'$ and $\lambda(p) = \lambda(q)$. Note that any path in the quotient $G/{\equiv_S}$ admits such a lift in $G$. If $q=(e_0,\dots, e_k)$ lifts $p$, we put $D(q) = \{ 0\leq i < k: \omega(e_i) \neq \alpha(e_{i+1})\}$. Clearly, $q$ is a path if and only if $D(q)$ is empty. Assume $j=\max(D(q))$, and consider a trivially-labeled path $r\colon \omega(e_j)\to \alpha(e_{j+1})$ in $G$, which we know exists by Part~(1). Let $q=q_1q_2$ be the factorization of $q$ where $|q_1|=j+1$. Then, $q'=q_1rq_2$ is an $\equiv_S$-path between $x'$ and $y'$ satisfying $\lambda(q')=\lambda(q)$ and $D(q')=D(q)\setminus\{j\}$. Thus, we may assume that $q$ is a path. Composing on both ends with trivially-labeled paths $x\to x'$ and $y'\to y$, we get a lift of $p$ which is a path between $x$ and $y$ in $G$. This result applied to loops shows that $\equiv_S$ is group-preserving. 

    Finally, let us suppose that $\equiv$ is another group-preserving congruence, and let $x\equiv y$. Choose any path $p\colon x\to y$. Then $p/{\equiv}$ is a loop over $y/{\equiv}$ in $G/{\equiv}$. Since $\equiv$ is group-preserving, there is a loop $q$ over $y$ with $\lambda(q)=\lambda(p/{\equiv})=\lambda(p)$. It follows that $pq^{-1}$ is a trivially-labeled path between $x$ and $y$, so $x\equiv_Sy$. 
\end{proof}

From now on, we will use the three equivalent descriptions of $\equiv_S$ interchangeably.

\section{Rauzy graphs}
\label{s:rauzy}

Recall that we defined the two maps $\init$ and $\tail$ by $\init(x) = x[0:|x|-1]$ and $\tail(x) = x[1:|x|]$. For $k\in\nn$, let us also define the map $\eval_k$ by $\eval_k(x) = x(k)$. Note that $\init$ and $\tail$ are defined on $A^+$, while $\eval_k$ is defined on $A^{>k}$.

\begin{definition}
    Let $L$ be a recurrent language on $A$ and $m, k\in\nn$ with $k\leq m$. The \emph{$k$-labeled Rauzy graph} of level $m$ of $L$ is the digraph $G_{m,k}$ defined by the diagram:
    \begin{center}
        \begin{tikzcd}
            & L\cap A^{m+1} \arrow[dl,"\eval_k",swap] \arrow[dr,shift left,"\init"] \arrow[dr,shift right,swap,"\tail"]& \\
            A & & L\cap A^m
        \end{tikzcd}
    \end{center}
\end{definition}

Special cases of these labeled Rauzy graphs have appeared in the litterature, including in \cite{Berthe2015} with $k=m$, and in \cite{Almeida2016} with $m=2k$. 

The maps $\init$, $\tail$ and $\eval_k$ used to define the Rauzy graphs are jointly injective, and moreover the following diagrams commute:
\begin{center}
    \begin{tikzcd}
        A^{\geq 2} \rar{\init} \dar{\tail} & A^+\dar{\tail}\\
        A^+ \rar{\init}& A^*
    \end{tikzcd}\qquad
    \begin{tikzcd}
        A^{>k+1} \rar{\init}\drar[swap]{\eval_k}& A^{>k}\dar{\eval_k}\\
         & A
    \end{tikzcd}\qquad
    \begin{tikzcd}
        A^{>k+1} \rar{\tail}\drar[swap]{\eval_{k+1}}& A^{>k}\dar{\eval_k}\\
         & A
    \end{tikzcd}
\end{center}

Therefore, $\init$ and $\tail$ also define onto digraph morphisms for $m\geq1$:
\begin{align*}
    \init\colon& G_{m,k}\to G_{m-1,k} & (0\leq k\leq m-1)\\
    \tail\colon& G_{m,k}\to G_{m-1,k-1} & (1\leq k\leq m).
\end{align*}

These morphisms will allow us to relate the groups defined the Rauzy graphs. In the next definition, we introduce a convenient notation for these groups.

\begin{definition}
    Let $(u,v)$ be such that $uv\in L$, $|u|=k$ and $|u|+|v|=m$. We denote by $H_{u,v}$ the group of $G_{m,k}$ at $uv$. We call $H_{u,v}$ a \emph{Rauzy group} of $L$. 
\end{definition}

The fact that $\tail$ and $\init$ define digraph morphisms immediately implies that: 
\begin{equation*}
    H_{u,v}\leq H_{\tail(u),v}, \quad H_{u,v}\leq H_{u,\init(v)}.
\end{equation*}
We further note that $H_{ua,v} = a^{-1}H_{u,av}a$. Since we are assuming that $L$ is recurrent, the Rauzy graphs are connected and it follows that $H_{u,v}$ and $H_{u',v'}$ lie in the same conjugacy class whenever $|uv| = |u'v'|$.

\section{Paths in suffix extension graphs}
\label{s:paths}

For this section, it is useful to introduce a local version of suffix-connectedness. We do this in the next definition.
\begin{definition}
    Let $m,e\in\nn$ with $1\leq e\leq m+1$. We say that $L$ is \emph{$(m,e)$-suffix-connected} if for every $w\in L\cap A^m$, there exists $1\leq d\leq e$ such that the natural embedding of $\lext(w)$ in $\ext_{d,d}(\tail^{d-1}(w))$ lies in a single connected component.
\end{definition}

\begin{remark}\label{r:local}
    This local version of suffix-connectedness has the following feature: suppose that $1\leq e\leq e'\leq m+1$ and that $L$ is $(m,e)$-suffix-connected; then $L$ is also $(m,e')$-suffix-connected. In particular, if we suppose that $L$ is suffix-connected, then it must be $(m,m+1)$-suffix-connected for all $m\geq 1$.
\end{remark}

The main result of this section is the following proposition, which is the main ingredient in the proof of Theorem~\ref{t:main}:
\begin{proposition}\label{p:paths}
    Assume that $L$ is recurrent and $(m-1,e)$-suffix-connected, where $m\geq 1$ and $1\leq e\leq m$. Then, $\ker(\tail)$ is a group-preserving equivalence relation of $G_{m,k}$ whenever $e\leq k\leq m$. 
\end{proposition}

The proof relies on the following lemma:
\begin{lemma}\label{l:rauzy1}
    Let $L$ be a recurrent language on $A$, $m\in\nn$, $0\leq k\leq m$, $d\geq 1$ and $x\in L\cap A^{m+d}$. Then there exists a path $p_x$ in $G_{m,k}$ such that:
    \begin{center}
        \begin{tikzcd}[column sep=large]
            p_x\colon\init^d(x)\rar{x[k:k+d]}&\tail^d(x)
        \end{tikzcd}
    \end{center}
\end{lemma}

\begin{proof}
    We proceed by induction on $d$. If $d=1$, then $x$ itself is an edge in $G_{m,k}$ providing the required path. 

    For the induction step, we assume that $d>1$. Let $x' = \init(x)$ and $x'' = \tail^{d-1}(x)$. Note that $|x''|=|x|-d+1 = m+1$, so $x''$ is an edge in $G_{m,k}$, which we see as a path of length~1. Moreover, the induction hypothesis gives us a path $p'$ such that
    \begin{center}
        \begin{tikzcd}[column sep = huge]
            p'\colon\init^{d-1}(x')\rar{x'[k:k+d-1]}&\tail^{d-1}(x')
        \end{tikzcd}
    \end{center}
    Recalling that $\init$ and $\tail$ commute, we find that:
    \begin{equation*}
        \tail^{d-1}(x') = \tail^{d-1}\circ\init(x) = \init\circ\tail^{d-1}(x) = \init(x'').
    \end{equation*}
    Hence, $p'$ and $x''$ are consecutive, and we may form the composition $p=p'x''$. Note that $p$ is a path between $\init^d(x)$ and $\tail^d(x)$, as required. Moreover, the label of this path is given by:
    \begin{equation*}
        x'[k:k+d-1]x''(k) = x[k:k+d-1]x(k+d-1) = x[k:k+d],
    \end{equation*}
    and this concludes the proof.
\end{proof}

We are now ready to prove the proposition above.
\begin{proof}[Proof of Proposition~\ref{p:paths}]
    Let us fix a pair of vertices identified by the digraph morphism $\tail:G_{m,k}\to G_{m-1,k-1}$, that is to say two words $ax, bx\in L\cap A^{m}$ where $x\in L\cap A^{m-1}$ and $a,b\in A$. We want to show $ax\equiv_Sbx$, which amounts to find a trivially-labeled path in $G_{m,k}$ between $ax$ and $bx$. 
    
    By assumption, there exists $d\leq e$ such that the natural embedding of $\lext(x)$ in $\ext_{d,d}(\tail^{d-1}(x))$ lies in one connected component. Let us write $y = \tail^{d-1}(x)$, and let $u$, $v$ be the natural embeddings of $a,b\in\lext(x)$ inside $\ext_{d,d}(y)$. In other words, $u$ and $v$ satisfy $ax=uy$ and $bx=vy$. Let us consider a path in $\ext_{d,d}(y)$ joining $u$ and $v$. Since $\ext_{d,d}(y)$ is bipartite, this path must have the following form:
    \begin{equation*}
        u=s_0, t_0, s_1, t_1, \dots, t_{j-1}, s_j = v,
    \end{equation*}
    where $s_i\in\lext_d(y)$, $t_i\in\rext_d(y)$. The fact that this forms a path in $\ext_{d,d}(y)$ means that, for each $0\leq i<j$, we have:
    \begin{equation*}
        s_iyt_i, s_{i+1}yt_i\in L.
    \end{equation*}
    Let us put $w_i = s_iyt_i$ and $z_i = s_{i+1}yt_i$. By Lemma \ref{l:rauzy1}, there exist paths:
    \begin{center}
        \begin{tikzcd}[column sep = large]
            p_i\colon\init^{d}(w_i)\rar{w_i[k:k+d]}&\tail^{d}(w_i)
        \end{tikzcd}\qquad
        \begin{tikzcd}[column sep = large]
            q_i\colon\init^{d}(z_i)\rar{z_i[k:k+d]}&\tail^{d}(z_i)
        \end{tikzcd}
    \end{center}
    We notice that $\init^d(w_i) = s_iy$, $\tail^d(w_i) = yt_i = \tail^d(z_i)$, $\init^d(z_i) = s_{i+1}y$. Therefore, $p_i, q_i^{-1}$ are consecutive and their composition is a path $s_iy\to s_{i+1}y$. Moreover, since $k\geq e\geq d$, it follows that 
    \begin{equation*}
        w_i[k:k+d] = (yt_i)[k-d:k] = z_i[k:k+d].
    \end{equation*}
    Therefore, $p_iq_i^{-1}$ is trivially-labeled. Composing these paths for $i=0,\dots, j-1$ gives us a trivially-labeled path between $ax=uy=s_0y$ and $s_jy=vy=bx$.
\end{proof}

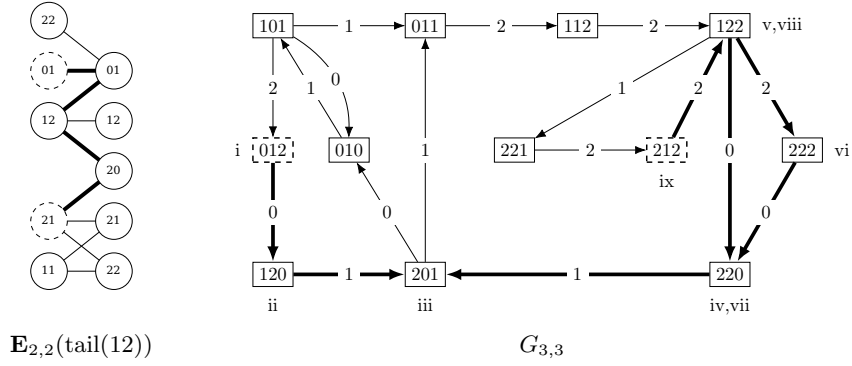
\begin{figure}\centering
    \begin{minipage}[b]{.25\linewidth}\centering
        \includegraphics[scale=.3]{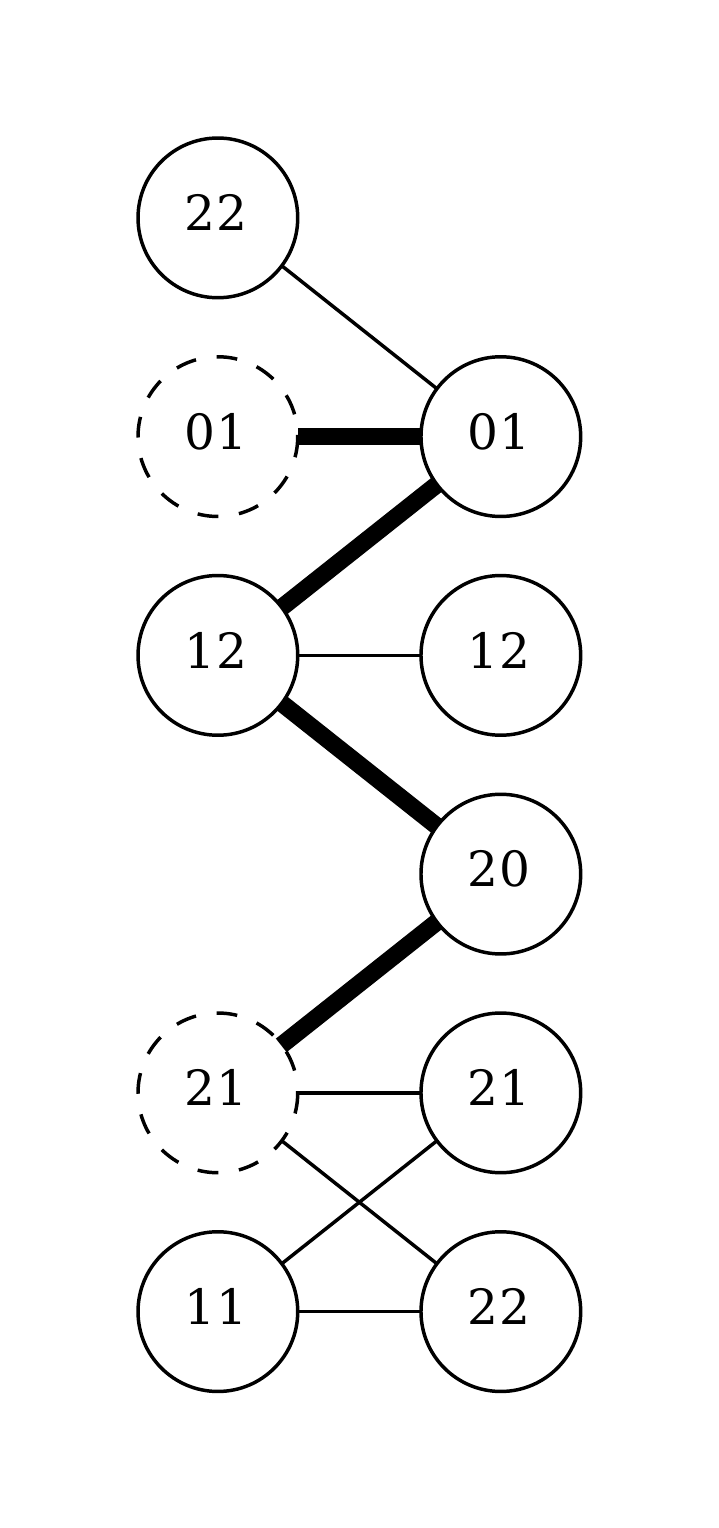}\\
        {\small$\ext_{2,2}(\tail(12))$}
    \end{minipage}
    \begin{minipage}[b]{.7\linewidth}\centering
        \scalebox{.7}{
            \begin{tikzpicture}
                \tikzstyle{every node}= [fill=white]%
                \node (101) at (63.0bp,219.0bp) [draw] {$101$};
                \node (011) at (145.0bp,219.0bp) [draw] {$011$};
                \node (122) at (309.0bp,219.0bp) [draw] {$122$};
                \node (112) at (227.0bp,219.0bp) [draw] {$112$};
                \node (012) at (63.0bp,152.0bp) [draw,dashed,thick] {$012$};
                \node (212) at (275.0bp,152.0bp) [draw,dashed,thick] {$212$};
                \node (010) at (104.0bp,152.0bp) [draw] {$010$};
                \node (221) at (193.0bp,152.0bp) [draw] {$221$};
                \node (222) at (348.0bp,152.0bp) [draw] {$222$};
                \node (220) at (309.0bp,85.0bp) [draw] {$220$};
                \node (120) at (63.0bp,85.0bp) [draw] {$120$};
                \node (201) at (145.0bp,85.0bp) [draw] {$201$};
                \node [left = 1mm of 012] (i) {i}; 
                \node [below = 1mm of 120] (ii) {ii}; 
                \node [below = 1mm of 201] (iii) {iii}; 
                \node [below = 1mm of 220] (iv) {iv,vii}; 
                \node [right = 1mm of 122] (v) {v,viii}; 
                \node [right = 1mm of 222] (vi) {vi}; 
                \node [below = 1mm of 212] (ix) {ix}; 
                \draw [->,line width=2pt,-{Latex[scale=.7]}] (012) -- node {$0$} (120);
                \draw [->,line width=2pt,-{Latex[scale=.7]}] (212) -- node {$2$} (122);
                \draw [->,-{Latex[scale=1.2]}] (101) -- node {$2$} (012);
                \draw [->,-{Latex[scale=1.2]}] (101) -- node {$1$} (011);
                \draw [->,-{Latex[scale=1.2]}] (101) to[bend left] node {$0$} (010);
                \draw [->,-{Latex[scale=1.2]}] (011) -- node {$2$} (112);
                \draw [->,-{Latex[scale=1.2]}] (010) -- node {$1$} (101);
                \draw [->,line width=2pt,-{Latex[scale=.7]}] (120) -- node {$1$} (201);
                \draw [->,-{Latex[scale=1.2]}] (201) -- node {$1$} (011);
                \draw [->,-{Latex[scale=1.2]}] (201) -- node {$0$} (010);
                \draw [->,-{Latex[scale=1.2]}] (122) -- node {$1$} (221);
                \draw [->,-{Latex[scale=.7]},line width=2pt] (122) -- node {$0$} (220);
                \draw [->,-{Latex[scale=.7]},line width=2pt] (122) -- node {$2$} (222);
                \draw [->,-{Latex[scale=1.2]}] (221) -- node {$2$} (212);
                \draw [->,line width =2pt,-{Latex[scale=.7]}] (220) -- node {$1$} (201);
                \draw [->,line width=2pt,-{Latex[scale=.7]}] (222) -- node {$0$} (220);
                \draw [->,-{Latex[scale=1.2]}] (112) -- node {$2$} (122);
            \end{tikzpicture}
        }\\
        {\small$G_{3,3}$}
    \end{minipage}
    \caption{A trivially-labeled path between a pair of words in $\ker(\tail)$ induced by a path in a depth 2 suffix extension graph. The Roman numerals indicate the order in which the vertices are visited in the Rauzy graph. This takes place in the language defined by the primitive substitution $0\mapsto 12, 1\mapsto 2, 2\mapsto 01$.}
\end{figure}

By combining Proposition \ref{p:paths} with Remark \ref{r:local}, it then follows that for each $m\geq 1$, the map $\tail$ defines a group-preserving morphism:
\begin{equation*}
    \tail\colon G_{m,m}\to G_{m-1,m-1}.
\end{equation*}
But clearly, the class of all group-preserving morphisms is closed under composition. Therefore, in a suffix-connected language, the following is a group-preserving morphism for all $m\geq 1$:
\begin{equation*}
    \tail^{m-1}\colon G_{m,m}\to G_{1,1}.
\end{equation*}
We immediately deduce the following:
\begin{corollary}\label{c:suff-connected}
    Let $L$ be a suffix-connected recurrent language and $u\in L$ with $u\neq\emptyw$. Then $H_{u,\emptyw} = H_{b,\emptyw}$, where $b$ is the last letter of $u$.
\end{corollary}

Let us highlight another particular case of this result. The condition of being $(m,1)$-suffix-connected is equivalent to being $m$-connected, meaning that $\ext(w)$ is connected for all $w\in L\cap A^m$, which in turn is equivalent to the dual condition of being $(m,1)$-prefix-connected. Combining Proposition \ref{p:paths} with its dual for the special case $e=1$, we obtain the following result, which is reminescent of \cite[Proposition~4.2]{Berthe2015}:
\begin{corollary}\label{c:connected}
    If $L$ is a $(m-1)$-connected recurrent language, where $m\geq 1$, then:
    \begin{enumerate}
        \item For $0\leq k\leq m-1$, $\ker(\init)$ is a group-preserving equivalence relation of $G_{m,k}$;
        \item For $1\leq k\leq m$, $\ker(\tail)$ is a group-preserving equivalence relation of $G_{m,k}$.
    \end{enumerate}
\end{corollary}

\section{Return sets}
\label{s:return}

Let us recall that the \emph{return set to $(u,v)$ in $L$} is the set of all words $r\in L$ such that $urv\in L$, $urv$ starts and ends with $uv$, and contains exactly two occurrences of $uv$. We denote this set by $\ret_{u,v}$. For basic properties of return sets, see \cite{Durand1998}. 
\begin{definition}
    Let $(u,v)$ be such that $uv\in L$. The subgroup of $F(A)$ generated by $\ret_{u,v}$ is denoted by $K_{u,v}$. We call this a \emph{return group} of $L$. 
\end{definition}

Our main result for this section relates the return groups with the Rauzy groups.
\begin{proposition}\label{p:zigzag}
    Let $L$ be a recurrent language and let $u,v$ be such that $uv\in L$.
    \begin{enumerate}
        \item $K_{u,v}\leq H_{u,v}$.
        \item If $\ret_{u,v}$ is finite and $s$ is one of its longest elements, then $H_{u,sv}\leq K_{u,v}$.
    \end{enumerate}
\end{proposition}

The following lemma recalls several properties of Rauzy graphs that will be relevant. By a \emph{positive path}, we mean a path which consists only of edges in $\edge(G)$ or, equivalently, which contains no formal inverses.
\begin{lemma}\label{l:rauzy2}
    Let $L$ be a recurrent language and let $u,v$ be such that $uv\in L$, $|u|=k$ and $|u|+|v|=m$.
    \begin{enumerate}
        \item Any element $w\in L$ is the label of a positive path in $G_{m,k}$.
        \item Any label $w$ of a positive path in $G_{m,k}$ of length at most $m+1$ is in $L$.
        \item If $p\colon x\to uv$ is a positive path in $G_{m,k}$, then $\lambda(p)$ is suffix-comparable with $u$. Moreover, there is at least one such path satisfying $\lambda(p)=u$.
        \item If $q\colon uv\to y$ is a positive path in $G_{m,k}$, then $\lambda(q)$ is prefix-comparable with $v$. Moreover, there is at least one such path satisfying $\lambda(q)=v$.
    \end{enumerate}
\end{lemma}
All of these properties follow from the definition of $G_{m,k}$ in a straightforward manner. Parts (1) and (2) are standard and can be found for instance in \cite[Section~4.1]{Berthe2015}. Parts (3) and (4) are analogous to \cite[Lemma~4.5]{Almeida2016}. 

We are now ready to prove the proposition. Let us mention that Part (2) of the proposition is inspired by the proof of \cite[Theorem~4.7]{Berthe2015}, which relied partly on the fact that return sets of the form $\ret_{u,\emptyw}$ are prefix codes. However, this property no longer holds for general return sets and we had to find a way to avoid it. This is essentially what is accomplished by the very last paragraph of the proof.
\begin{proof}[Proof of Proposition~\ref{p:zigzag}]
    \textbf{(1)} Let $k=|u|$, $m=|u|+|v|$, and fix $r\in\ret_{u,v}$. Then $urv\in L$ is the label of a positive path $p$ in $G_{m,k}$ by Part (1) of Lemma~\ref{l:rauzy2}. Consider the factorization $p = q_1p'q_2$, where $\lambda(q_1)=u$, $\lambda(p')=r$ and $\lambda(q_2)=v$. Write $\alpha(p')=x_1$ and $\omega(p')=x_2$. Consider the factorization $x_1 = u_1v_1$, where $|u_1| = k$. Since $\omega(q_1) = x_1$, it follows from Part (3) of Lemma~\ref{l:rauzy2} that $u_1$ is suffix-comparable with $\lambda(q_1) = u$. As $|u|=k=|u_1|$, we conclude that $u_1 = u$. Similarly, $\alpha(p'q_2) = x_1$, so Part (4) implies that $v_1$ is prefix-comparable with $\lambda(p'q_2) = rv$. Since $r\in\ret_{u,v}$, the word $rv$ starts with $v$, and since $|v| = m-k = |v_1|$, we conclude that $v_1 = v$. Thus, $x_1 = uv$. A similar argument yields $x_2=uv$, so $p'$ is a loop over $uv$, and $\ret_{u,v}\subseteq H_{u,v}$. This proves (1).

    \textbf{(2)} Let $m'=|u|+|s|+|v|$, and consider a positive path in $G_{m',k}$ of the form $p\colon uvx\to uvy$. We start by proving the following claim: $w=\lambda(p)$ is a concatenation of elements of $\ret_{u,v}$. 

    To prove this claim, let us first consider two positive paths $q_1\colon x_1\stackrel{u}{\to}uvx$ and $q_2\colon uvy\stackrel{v}{\to}x_2$, whose existence is a consequence of Part (3) and (4) of Lemma~\ref{l:rauzy2}.  Since $\alpha(pq_2) = uvx$, Part (4) of Lemma~\ref{l:rauzy2} implies that $\lambda(pq_2)$ is prefix-comparable with $vx$; thus, it starts with $v$. Similarly, since $\omega(q_1p) = uvy$, Part(3) of Lemma~\ref{l:rauzy2} implies that $\lambda(q_1p)$ is suffix-comparable with $u$; thus, it ends with $u$. In particular, this implies  
    \begin{equation*}
        uwv=\lambda(q_1)\lambda(pq_2)=\lambda(q_1p)\lambda(q_2)\in uvA^*\cap A^*uv.
    \end{equation*}

    We now prove the claim by induction on $|w|=|p|$. If $|w|\leq |s|$, then $uwv$ is the label of the positive path $q_1pq_2$ in $G_{m',k}$, which has length at most $m'$. Hence, Part~(2) of Lemma~\ref{l:rauzy2} implies that $uwv\in L$. This, taken together with the fact that $uwv$ belongs to $uvA^*\cap A^*uv$, implies that $w$ is a concatenation of elements of $\ret_{u,v}$. This establishes the basis of the induction.

    For the inductive step, let us suppose that $|w|>|s|$. Let $p'$ be the prefix of $p$ of length $m'$ of $q_1pq_2$, and let $z = \lambda(p')$. By Part (2) of Lemma~\ref{l:rauzy2}, $z\in L$. Moreover, $z \in uvA^*$, so it is prefix-comparable with some element of $u\ret_{u,v}v$. But by assumption, $|z|$ is the maximal length of an element of $u\ret_{u,v}v$. Therefore, it follows that $z$ has at least two occurrences of $uv$. Since $z$ is a proper prefix of $uwv$, we deduce that $uwv$ has an occurrence of $uv$ at position $0<j<|s|$. Consider the factorization $p = p_1p_2$ where $|p_1|=j$, and let $x'=\omega(p_1)=\alpha(p_2)$. Since $j$ is an occurrence of $uv$ in $uwv = \lambda(q_1pq_2)$ and $|q_1p_1| = |u|+j$, it follows that $\lambda(q_1p_1)$ ends with $u$. Consider the factorization $x' = u'x''$, where $|u'|=|u|$. By Part (3) of Lemma~\ref{l:rauzy2}, $u'$ is suffix-comparable with $\lambda(q_1p_1)$, and since $|u'|=|u|$, it follows that $u' = u$. Similarly, the fact that $j$ is an occurrence of $uv$ in $uwv$, with $uwv = \lambda(q_1p_1p_2q_2)$ and $|q_1p_1| = |u|+j$, implies that $\lambda(p_2q_2)$ starts with $v$. By Part (4) of Lemma~\ref{l:rauzy2}, it follows that $x''$ is prefix-comparable with $\lambda(p_2q_2)$, and hence with $v$. However, recall that $x'\in L\cap A^{m'}$ where $m' = |u|+|s|+|v|$:
    \begin{equation*}
        |x''| = |x'| - |u| = |s|+|v| \geq |v|.
    \end{equation*}
    Therefore, $v$ is a prefix of $x''$, and $x''=vt$ for some word $t$. Hence, we conclude that $p_1,p_2$ satisfy:
    \begin{equation*}
        p_1\colon uvx\to uvt,\quad p_2\colon uvt\to uvy.
    \end{equation*}
    Since $0<j<|s|<|w|$, we have $|p_1|<|p|$ and $|p_2|<|p|$. Thus, by the induction hypothesis, both $\lambda(p_1)$ and $\lambda(p_2)$ are product of words in $\ret_{u,v}$. And, therefore, so is $w$. This finishes the proof of the claim.

    We finish the proof of Part (2) of the proposition by showing that it follows from that claim. First, recall that $G_{m',k}$ is \emph{strongly connected}, in the sense that any two vertices can be joined by a positive path. Moreover, the groups of a strongly connected digraph are generated by the labels of positive loops \cite[Corollary~4.5]{Steinberg2000}. Since the claim above shows in particular that the labels of positive loops over $usv$ in $G_{m',k}$ lie in $K_{u,v}$, the result follows.
\end{proof}

\section{Proof of the main result}
\label{s:main}

Let us first recall the statement of our main result, Theorem~\ref{t:main}: if $L$ is a suffix-connected uniformly recurrent language on $A$, then all the return groups of $L$ lie in the same conjugacy class and their rank is $n-c+1$, where $n = \Card(A)$ and $c$ is the number of connected components of $\ext(\emptyw)$.

The proof is split in two lemmas. In the first one, we apply the results obtained in the previous sections to show that (under the assumptions of Theorem~\ref{t:main}) all the return groups of $L$ belong to the same conjugacy class. The second lemma finishes the proof by showing that the groups in this conjugacy class have rank $n-c+1$.

\begin{lemma}\label{l:level1conjugacy}
    Let $L$ be a uniformly recurrent suffix-connected language. Then, the return groups of $L$ lie in the conjugacy class of subgroups of $F(A)$ generated by the groups of the Rauzy graph $G_{1,1}$.
\end{lemma}

\begin{proof}
    Consider a pair $(u,v)$ such that $uv\in L$ and $uv\neq\emptyw$. By Corollary~\ref{c:suff-connected}, $H_{uv,\emptyw} = H_{b,\emptyw}$, where $b$ is the last letter of $uv$. Using the conjugacy relation between Rauzy groups, we then have:
    \begin{equation*}
        H_{u,v} = vH_{uv,\emptyw}v^{-1} = vH_{b,\emptyw}v^{-1},
    \end{equation*}
    and this equality holds for any such pair $(u,v)$.

    Since $L$ is uniformly recurrent, we may choose an element $s\in\ret_{u,v}$ of maximum length and by Proposition~\ref{p:zigzag}:
    \begin{equation*}
        H_{u,sv} \leq K_{u,v} \leq H_{u,v}.
    \end{equation*}
    Applying the conclusion of the previous paragraph to the pair $u,sv$ while noting that $s\in K_{u,v} \leq H_{u,v}$, we get:
    \begin{equation*}
        H_{u,sv} = svH_{b,\emptyw}v^{-1}s^{-1} = sH_{u,v}s^{-1} = H_{u,v}.
    \end{equation*}
    Hence, $K_{u,v} = H_{u,v} = vH_{b,\emptyw}v^{-1}$. 
\end{proof}

The next lemma concludes the proof of Theorem~\ref{t:main}. It also gives an effective way of computing the Stallings equivalence of $G_{1,1}$ and, in turn, the Stallings equivalence can be used to find a basis for any of the groups defined by $G_{1,1}$. 
\begin{lemma}\label{l:level1groups}
    Let $L$ be a recurrent language. Then, the groups of the Rauzy graph $G_{1,1}$ have rank $n-c+1$, where $n=\Card(A)$ and $c$ is the number of connected components of $\ext(\emptyw)$.
\end{lemma}

\begin{proof}
    A well-known consequence of Stallings algorithm is that the rank of any group generated by a connected digraph $G$ is
    \begin{equation*}
        \Card(\edge(G/{\equiv_S}))-\Card(\vertex(G/{\equiv_S})) + 1
    \end{equation*}
    (see \cite[Lemma~8.2]{Kapovich2002}). Thus, we need only to show that the quotient $G_{1,1}/{\equiv_S}$ has $c$ vertices and $n$ edges.

    Let us start by showing that $G_{1,1}/{\equiv_S}$ has $c$ vertices. By definition, we have $\vertex(G_{1,1}) = A = \lext(\emptyw)$. Let $\sim$ be the relation defined as follow: for $a,b\in A$, we have $a\sim b$ exactly when, viewed as elements of $\lext(\emptyw)$, $a$ and $b$ lie in the same connected component of $\ext(\emptyw)$. Note that the relation $\sim$ has precisely $c$ classes because every connected component of $\ext(\emptyw)$ contains at least one vertex in $\lext(\emptyw)$. Therefore, it suffices to show that ${\sim} = {\equiv_S}$. 

    We now prove the inclusion ${\sim}\subseteq{\equiv_S}$. Since $\ext(\emptyw)$ is bipartite, any path in $\ext(\emptyw)$ between elements of $\lext(\emptyw)$ has even length, and thus it suffices to argue for elements related by paths of length~2. Let us assume that $a,b\in\lext(\emptyw)=A$ are related by a path of length~2 inside $\ext(\emptyw)$. By definition of $\ext(\emptyw)$, the existence of such a path means that there is some $c\in A$ such that $ac, bc\in L$. But recall that $\edge(G_{1,1}) = L\cap A^2$, so we may view $e = ac$ and $f= bc$ as edges in $G_{1,1}$, both of which have label $c$. The path $(e, f^{-1})$ is then a trivially-labeled path between $a$ and $b$, so $a\equiv_S b$ as required.

    Let us prove the inclusion ${\equiv_S}\subseteq{\sim}$. By definition of $\equiv_S$, we only need to show that $\sim$ is closed under the two rules (F) and (F\textquotesingle). We argue for each separately. Let us fix $a,b,c,c',d,d'\in A$ such that $a\sim c'$, $b\sim d'$, $c\sim d$.
    \begin{itemize}
        \item[(F)] 
            We assume that there are edges $e\colon c\to c'$, $f\colon d\to d'$ such that $\lambda(e) = \lambda(f)$. Note that the maps $\tail$ and $\eval_1$ agree on $A^2$, so by definition $\omega = \lambda$ in the Rauzy graph $G_{1,1}$. Thus, under our current assumptions, $c' = d'$. Hence, $a\sim c' \sim b$ and $a\sim b$ by transitivity.
        \item[(F\textquotesingle)]
            We assume that there are edges $e\colon c'\to c$, $f\colon d'\to d$ such that $\lambda(e) = \lambda(f)$. Since $\omega = \lambda$ in $G_{1,1}$, we deduce that $d = c$. By definition, $\edge(G_{1,1}) = L\cap A^2$, $e = c'c\in L$ and $f = d'd = d'c\in L$. In particular, there is an edge in $\ext(\emptyw)$ joining $c'$ and $c$, and another one joining $d'$ and $c$. Hence, $a\sim c'\sim d'\sim b$ and $a\sim b$ by transitivity.        
    \end{itemize}

    It only remains to show that $G_{1,1}/{\equiv_S}$ has $n$ edges. We do this by showing that the labeling map $\lambda\colon G_{1,1}/{\equiv_S}\to A$ is a bijection. Fix a letter $a\in A$. Since $L$ is recurrent, there exists $b\in A$ with $ba\in L$. Hence, there is at least one edge labeled $a$ in $G_{1,1}$, and therefore also in $G_{1,1}/{\equiv_S}$. Hence, $\lambda\colon G_{1,1}/{\equiv_S}\to A$ is surjective. Now suppose that $G_{1,1}$ has two edges $e,f$ labeled $a$. As noted before, $\lambda=\omega$ in $G_{1,1}$, so $\omega(e) = a = \omega(f)$. Applying rule (F\textquotesingle), we conclude that $\alpha(e)\equiv_S\alpha(f)$. In particular, $e/{\equiv_S} = f/{\equiv_S}$, which proves that the labeling map $\lambda\colon G_{1,1}/{\equiv_S}\to A$ is injective.
\end{proof}

\begin{figure}\centering
    \begin{minipage}[b]{0.4\linewidth}\centering
        \begin{tikzpicture}[bend angle=20,auto,scale=0.9]
            \node (c) at (0,0) {$c$};
            \node (d) at ([shift=({0:1.3 cm})]c) {$d$};
            \node (b) at ([shift=({90:1.3 cm})]d) {$b$};
            \node (a) at ([shift=({180:1.3 cm})]b) {$a$};
            \node (c') at ([shift=({135:1.3 cm})]a) {$c'$};
            \node (d') at ([shift=({45:1.3 cm})]b) {$d'$};
            \draw (c) -- (d); 
            \draw[ultra thick] (c) to node {$\sim$} (d); 
            \draw[ultra thick] (c') to node {$\sim$} (a); 
            \draw[ultra thick] (d') to[swap] node {$\sim$} (b); 
            \draw[dashed, thick] (b) -- (a); 
            \draw[->,>=latex,bend left] (c) to node {$d'=c'$} (c') ;
            \draw[->,>=latex,bend right] (d) to node[swap] {$d'=c'$} (d') ;
        \end{tikzpicture}\\
        {\small(F)}
    \end{minipage}
    \begin{minipage}[b]{0.4\linewidth}\centering
        \begin{tikzpicture}[bend angle=20,auto,scale=0.9]
            \node (c) at (0,0) {$c$};
            \node (d) at ([shift=({0:1.3 cm})]c) {$d$};
            \node (b) at ([shift=({90:1.3 cm})]d) {$b$};
            \node (a) at ([shift=({180:1.3 cm})]b) {$a$};
            \node (c') at ([shift=({135:1.3 cm})]a) {$c'$};
            \node (d') at ([shift=({45:1.3 cm})]b) {$d'$};
            \draw[ultra thick] (c) to node {$\sim$} (d); 
            \draw[ultra thick] (c') to node {$\sim$} (a); 
            \draw[ultra thick] (d') to node[swap] {$\sim$} (b); 
            \draw[dashed, thick] (b) -- (a); 
            \draw[->,>=latex,bend right] (c') to node[swap] {$d=c$} (c) ;
            \draw[->,>=latex,bend left] (d') to node {$d=c$} (d) ;
        \end{tikzpicture}\\
        {\small(F\textquotesingle)}
    \end{minipage}
    \caption{The rules (F) and (F\textquotesingle) as they appear in the proof of Lemma~\ref{l:level1groups} when showing that ${\equiv_S}\subseteq{\sim}$.}
\end{figure}
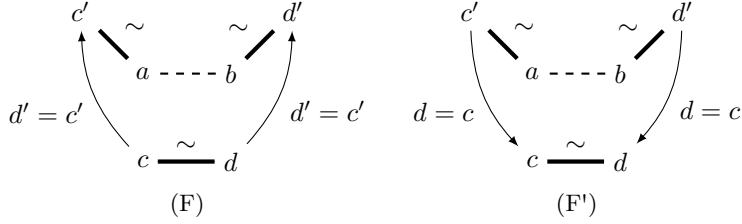

\section{Proof of the corollaries}
\label{s:cor}

Let us start this section by recalling the statement of Corollary~\ref{c:main1}: if $L$ is suffix-connected and uniformly recurrent, then the following statements are equivalent:
\begin{enumerate}
    \item All the return sets of $L$ generate the full free group $F(A)$.
    \item Some return set of $L$ generates a group of rank $\Card(A)$.
    \item The extension graph of the empty word is connected.
\end{enumerate}

\begin{proof}[Proof of Corollary~\ref{c:main1}]
    \textbf{(1) implies (2).} Trivial.

    \textbf{(2) implies (3).} By Theorem~\ref{t:main}, all return groups of $L$ have rank $n-c+1$ where $c$ is the number of connected components of $\ext(\emptyw)$ and $n = \Card(A)$. Under the assumption (2), we therefore have $n = n-c+1$ and $c=1$.

    \textbf{(3) implies (1).} If $\ext(\emptyw)$ is connected, then by Corollary~\ref{c:connected}, there is a group-preserving morphism $G_{1,1}\to G_{0,0}$. But note that $G_{0,0}$ has a single vertex with loops labeled by the letters of $A$. Thus, the group generated by $G_{0,0}$ is equal to the full free group $F(A)$, and so are all the groups of the level~1 Rauzy graph $G_{1,1}$. But recall that, for a suffix-connected language, all the return groups lie in the conjugacy class generated by the level~1 Rauzy groups (see Lemma \ref{l:level1conjugacy}), and so the result follows. 
\end{proof}

Before proving Corollary~\ref{c:main2}, we need some preliminary material. A word $w\in L$ is called \emph{neutral} if:
\begin{equation*}
    1-\chi(\ext(w)) = 0,
\end{equation*}
where $\chi(\ext(w))$, the characteristic of $\ext(w)$, is the difference between the number of vertices and edges in $\ext(w)$. A \emph{neutral language} is a language in which all non-empty words are neutral. The next result, quoted from \cite[Corollary~5.4]{Dolce2017}, will be useful to prove Corollary~\ref{c:main2}. 
\begin{lemma}\label{l:neutral}
    If $L$ is recurrent and neutral, then for all $u,v$ with $uv\in L$,
    \begin{equation*}
        \Card(\ret_{u,v}) = \Card(A)-\chi(\ext(\emptyw))+1.
    \end{equation*}
\end{lemma}

With this lemma in mind, let us recall the statement of Corollary~\ref{c:main2}: if $L$ is uniformly recurrent, connected and neutral, then the following statements are equivalent:
\begin{enumerate}
    \item Some return set of $L$ is a free subset of the free group $F(A)$.
    \item All return sets of $L$ are free subsets of the free group $F(A)$.
    \item $L$ is a tree set.
\end{enumerate}

\begin{proof}[Proof of Corollary~\ref{c:main2}]
    We first recall the following fact, which is a straightforward consequence of the well-known Hopfian property of $F(A)$: a finite subset of $F(A)$ is free if and only its cardinality agrees with the rank of the subgroup it generates. Moreover, note that under our assumptions, the return sets of $L$ all have the same cardinality (by Lemma~\ref{l:neutral}), as well as the same rank (by Theorem~\ref{t:main}). Therefore if \emph{some} return set is free, then \emph{all} return sets are free, that is to say (1) and (2) are equivalent.

    To prove the equivalence of (2) and (3), we use the following fact about graphs: a simple graph $G$ is a forest if and only if it has exactly $\chi(G)$ connected components \cite[Exercise~2.1.7~(b)]{Bondy1976}. On the one hand, this implies that in a connected neutral language, the extension graph of any non-empty word must be a tree (since neutrality implies $\chi(\ext(w)) = 1$). Thus, a neutral connected language is a tree set if and only if $\ext(\emptyw)$ is a forest, if and only if $\chi(\ext(\emptyw)) = c$, where $c$ denotes the number of connected components of $\ext(\emptyw)$. This is also equivalent to the following equality:
    \begin{equation*}
        \Card(A)-\chi(\ext(\emptyw))+1 = \Card(A)-c+1.
    \end{equation*}
    Let us fix $u, v$ with $uv\in L$. Lemma~\ref{l:neutral} implies that $\Card(\ret_{u,v})$ is equal to the left-hand side of the previous equation, while Theorem~\ref{t:main} implies that $\rank(K_{u,v})$ is equal to the right hand side. Since (2) holds exactly when $\Card(\ret_{u,v}) = \rank(K_{u,v})$ for all such $u,v$, the result follows.
\end{proof}

\section{Suffix-connected example}
\label{s:example}

This section is devoted to the proof of Theorem~\ref{t:example}. We consider the following substitution on the alphabet $A=\{0,1,2\}$:
\begin{equation*}
    \begin{array}{llll} 
        \varphi\colon & 0 & \mapsto & 0001 \\
                 & 1 & \mapsto & 02 \\ 
                 & 2 & \mapsto & 001 
    \end{array}.
\end{equation*}
Note that $\varphi$ is primitive (since for every $a,b\in A$, $a$ occurs in $\varphi^3(b)$), and that $\varphi(A)$ is a \emph{prefix code} (no word in $\varphi(A)$ is a proper prefix of another). In particular, this implies that $\varphi$ is injective, a fact that will be used several times. We recall that the \emph{language} defined by $\varphi$ is the subset of all words $w\in A^+$ such that $w$ is a factor of $\varphi^n(a)$ for some $a\in A$ and $n\in\nn$. For the current section, $L$ denotes the language of $\varphi$. As we already mentioned, it is well known that the language of a primitive substitution is uniformly recurrent. We will show that $L$ is suffix-connected, and deduce that all the return sets of $L$ generate the full free group $F(A)$. 

The proof, being a bit lengthy, is organized in \ref*{step:conclude} steps. Let us give a quick outline of each step:
\begin{enumerate}
    \item\label{step:classification} We show that every right special factor of length at least 2 either ends with 00 and satisfies $\rext(x) = \{0,1\}$; or ends with 10 and satisfies $\rext(x) = \{0,2\}$. Similarly, we show that every left special factor of length at least 3 either starts with 000 and satisfies $\lext(x) = \{1,2\}$; or starts with 001 and satisfies $\lext(x) = \{0,1\}$.
    \item\label{step:001} We show that $L$ contains only 4 bispecial factors starting with 001 and we compute them.
    \item\label{step:stability} We show that if $x$ is a bispecial factor that starts with $000$, then
        \begin{equation*}
            \ext(x) \isom
            \begin{cases}
                \ext(\varphi(x)0) & \text{if $x$ ends with 00},\\
                \ext(\varphi(x)00) & \text{if $x$ ends with 10}.
            \end{cases}
        \end{equation*}
    \item\label{step:wk} We define inductively a sequence of words $(w_k)_{k\in\nn}$ of increasing lengths, and we show that the disconnected elements of $L$ are precisely the members of that sequence.
    \item\label{step:conclude} We define a sequence of integers $(d_k)_{k\in\nn}$ such that $\lext(w_k)$ embeds in one connected component of $\ext_{d_k,d_k}(\tail^{d_k-1}(w_k))$. \label{i:last}
\end{enumerate}

Some of these steps involve the computation of the sets $L\cap A^k$ for several values of $k$, some of them quite large. We will omit the details of these computations and provide only the results. These computations can be checked either by hand (e.g.\ with the algorithm described in \cite[Section~3.2]{Balchin2017}), or perhaps more appropriately using SageMath \cite{Developers2020}. At the time of writing, a SageMath web interface can be accessed at the address \url{https://sagecell.sagemath.org}. To compute the set $L\cap A^k$, simply evaluate the following line of code in the web interface:
\begin{center}
    \texttt{WordMorphism(\{0:[0,0,0,1],1:[0,2],2:[0,0,1]\}).language($k$)}.
\end{center}

\subsection*{Step~\ref{step:classification}}
We prove the following claim. 
\begin{claim}
    Let $x$ be a right special factor of $L$ of length at least 2. Then one of the two following alternatives hold:
    \begin{enumerate}
        \item $x$ ends with $00$ and $\rext(x) = \{0,1\}$.
        \item $x$ ends with $10$ and $\rext(x) = \{0,2\}$.
    \end{enumerate}
    Dually, let $y$ be a left special factor of $L$ of length at least 3. Then one of the two following alternatives hold:
    \begin{enumerate}
        \item $y$ starts with $000$ and $\lext(y) = \{1,2\}$.
        \item $y$ starts with $001$ and $\lext(y) = \{0,1\}$.
    \end{enumerate}
\end{claim}

\begin{proof}[Proof of the claim]
    Direct computations reveal that:
    \begin{equation*}
        L\cap A^3 = \{ 000, 001, 010, 020, 100, 102, 200\}.
    \end{equation*}
    Hence, the only two right special factors in $L\cap A^2$ are $00$ and $10$, and they satisfy respectively:
    \begin{equation*}
        \rext(00) = \{0,1\},\qquad \rext(10) = \{0,2\}.
    \end{equation*}
    Since the sets $\rext(x)$ are weakly increasing under taking suffixes, the first part follows.

    Similarly, we find:
    \begin{equation*}
        L\cap A^4 = \{ 0001, 0010, 0100, 0102, 0200, 1000, 1001, 1020, 2000 \}.
    \end{equation*}
    Therefore $L\cap A^3$ contains only two left special factors, $000$ and $001$, satisfying respecitvely:
    \begin{equation*}
        \lext(000) = \{ 1,2 \},\qquad \lext(001) = \{ 0,1 \}.
    \end{equation*}
    Since the sets $\lext(x)$ are weakly increasing under taking prefixes, the second part follows as well.
\end{proof}

\subsection*{Step~\ref{step:001}}
We now know that all long enough bispecial factors must start with either $000$ or $001$, and end with either $00$ or $10$. We restrict the possibilities even further by proving the following claim: 
\begin{claim}
    The only four bispecial factors of $L$ starting with $001$ are:
    \begin{equation*}
        0010,\quad 00100,\quad 00100010,\quad 001000100010.
    \end{equation*}
\end{claim}

The proof of this claim makes use of the concept of \emph{cutting points}, by which we mean the following: in a word of the form $\varphi(z)$, a cutting point is an index $0\leq j\leq |\varphi(z)|-1$ such that $j = |\varphi(z_1)|$, for some prefix $z_1$ of $z$. We observe that in the specific case of $\varphi$, the cutting points are located exactly after the occurrences of the letters $1$ and $2$. The following elementary lemma  will be useful.

\begin{lemma}\label{l:cutting}
    Let $\varphi(z) = u_1\dots u_n$ be a factorization such that $|u_1\dots u_k|$ is a cutting point for all $1\leq k<n$. Then there is a factorization $z = z_1\dots z_n$ such that $\varphi(z_i) = u_i$ for all $1\leq i\leq n$.
\end{lemma}

\begin{proof}
    By assumption, $u_1\dots u_{j-1}, u_1\dots u_j\in\varphi(A^*)$ for all $1\leq j\leq n$. Since $\varphi(A)$ is a prefix code, we have $u_j\in\varphi(A^*)$ and we may write $u_j = \varphi(z_j)$ for some $z_j$. Then $\varphi(z)=\varphi(z_1\dots z_n)$ and as $\varphi$ is injective, $z = z_1\dots z_n$ as required.
\end{proof}

\begin{proof}[Proof of the claim]
    Let us start by noting that the only bispecial factors of $L$ with length at most 4 are:
    \begin{equation*}
        \emptyw,\quad 0,\quad 00,\quad 0010.
    \end{equation*}
    This can be proven simply by inspecting the sets $L\cap A^k$ for $2\leq k\leq 6$. The extension graphs of these four words can be found in Figure~\ref{f:ext-bisp-lt4}. From now on, we work only with bispecial factors of length at least 5.
    \begin{figure}\centering
        \begin{minipage}[b]{.23\linewidth}\centering
            \includegraphics[scale=.4]{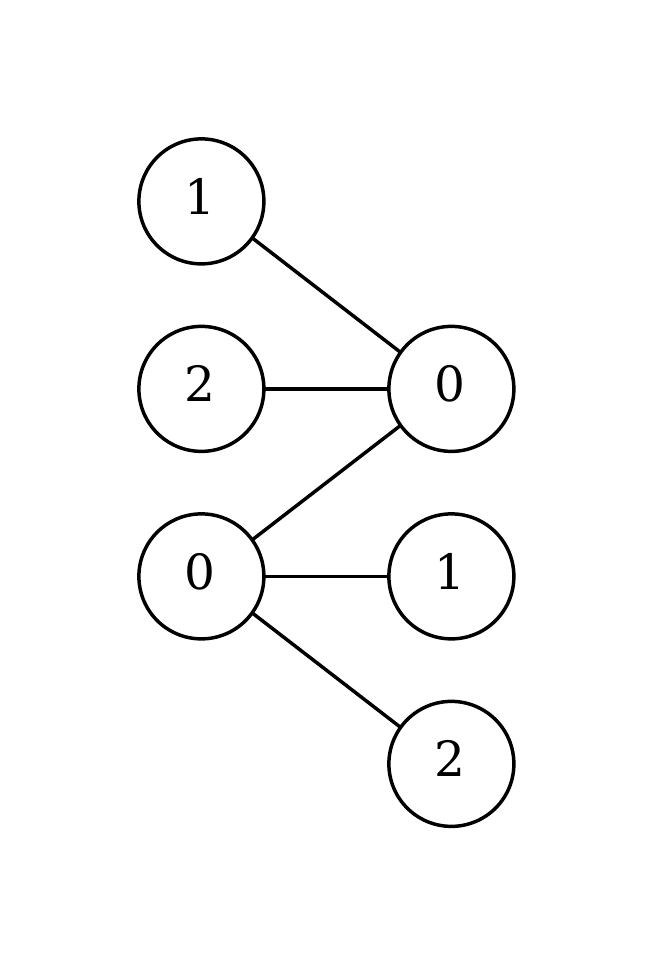}\\
            {\small$\ext(\emptyw)$}
        \end{minipage}
        \begin{minipage}[b]{.23\linewidth}\centering
            \includegraphics[scale=.4]{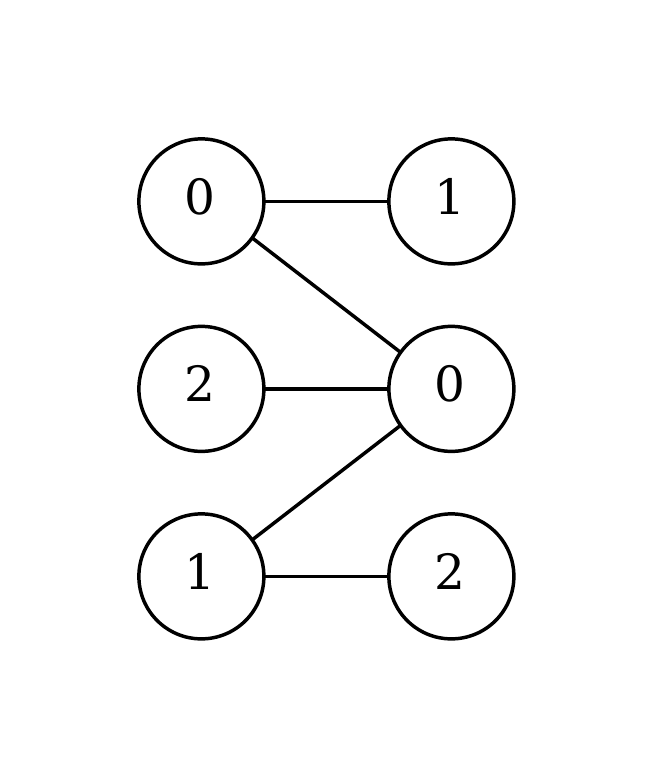}\\
            {\small$\ext(0)$}
        \end{minipage}
        \begin{minipage}[b]{.23\linewidth}\centering
            \includegraphics[scale=.4]{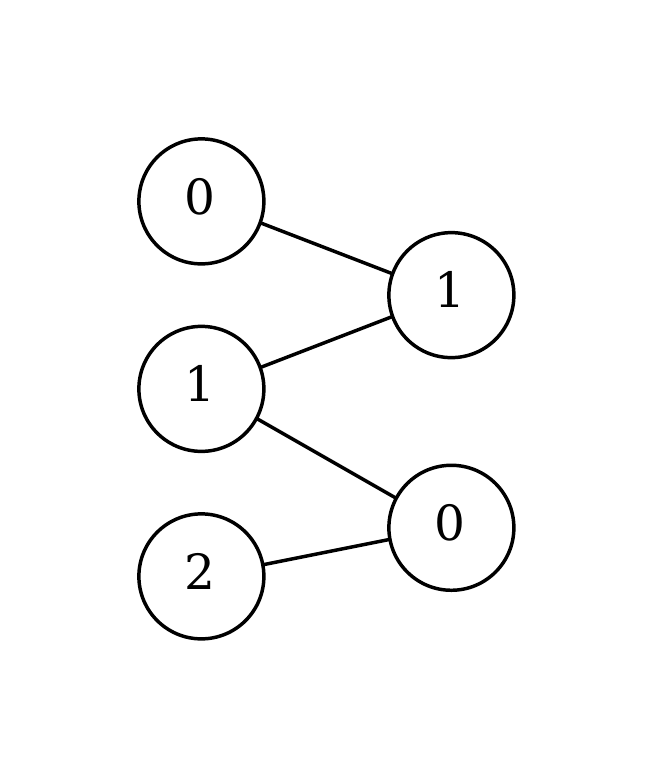}\\
            {\small$\ext(00)$}
        \end{minipage}
        \begin{minipage}[b]{.23\linewidth}\centering
            \includegraphics[scale=.4]{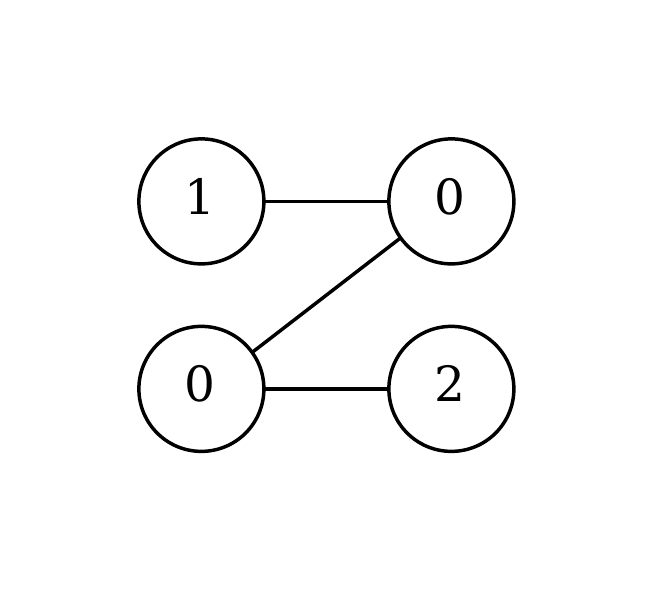}\\
            {\small$\ext(0010)$}
        \end{minipage}
        \caption{Extension graphs of all the bispecial words of length at most 4 in $L$.}\label{f:ext-bisp-lt4}
    \end{figure}

    Let us suppose that $u$ is bispecial, $|u|\geq 5$ and $001$ is a prefix of $u$. We distinguish two cases: $u = 001x10$ and $u = 001x00$.
    
    We start by the case $u=001x10$. Let $x' = x1$. By Step~\ref{step:classification}, we know that $0u, 1u\in L$. Thus, there exist $z_1, z_2\in L$ such that:
    \begin{equation*}
        \varphi(z_1) = s_10001x'0t_1,\qquad \varphi(z_2) = s_21001x'0t_2.
    \end{equation*}
    
    Since $0000\notin L$, it follows that $s_1$ ends with either 1 or 2. Therefore $|s_1|$ is a cutting point in $\varphi(z_1)$. Similarly, there is a cutting point in $\varphi(z_1)$ at the end of $x'$. It follows from Lemma~\ref{l:cutting} that $z_1$ has a factor of the form $0y_1$ such that $\varphi(y_1)=x'$. With similar arguments, we conclude that $z_2$ has a factor of the form $2y_2$ such that $\varphi(y_2)=x'$. Since $\varphi$ is injective, we find $y_1=y_2=y$, and $\lext(y) \supseteq\{0,2\}$. By Step~\ref{step:classification}, it follows that $|y|<3$, which leaves us with only thirteen possibilities. Further accounting for the fact that $2\in\lext(y)$ and $\varphi(y)\neq\emptyw$, we narrow it down to only two possibilities, namely $y=0$ and $y=00$. Trying out both values, we obtain either:
    \begin{align*}
        u &= 001\varphi(0)0 = 00100010; \text{ or}\\
        u &= 001\varphi(00)0 = 001000100010.
    \end{align*}
    A direct computation shows that both of those words are bispecial.

    Finally, we treat the case $u=001x00$. By Step~\ref{step:classification}, $u0 = 001x000\in L$. Since $0000\notin L$, it follows that $x$ cannot end with $0$. Moreover, we also have $0u, 1u, u1\in L$, so there exist $z_1, z_2, z_3\in L$ such that:
    \begin{equation*}
        \varphi(z_1) = s_10001x00t_1,\quad \varphi(z_2) = s_21001x00t_2,\quad \varphi(z_3) = s_3001x001t_3.
    \end{equation*}
    Again, since $0000\notin L$, $s_1$ cannot end with $0$. Recalling that cutting points are located exactly after the occurrences of 1 or 2, we apply Lemma~\ref{l:cutting} to conclude that there exist: a factor of $z_1$ of the form $0y_1$ such that $\varphi(y_1)=x$; a factor of $z_2$ of the form $2y_2$ such that $\varphi(y_2)=x$; and a factor of $z_3$ of the form $y_32$ such that $\varphi(y_3)=x$. Since $\varphi$ is injective, $y_1=y_2=y_3=y$, and $\lext(y)\supseteq\{0,2\}$, $2\in\rext(y)$. By Step~\ref{step:classification}, we conclude that $|y|<3$, which again leaves us with thirteen possible values for $y$. Accounting for the fact that $2\in\rext(y)$ and $2\in\lext(y)$ narrows this to only two possibilities: $y = \emptyw$ and $y = 0$. Testing both possibilities, we find that $y=0$ does not yield a bispecial factor, leaving us with only one bispecial factor for that case:
    \begin{equation*}
        u=001\varphi(\emptyw)00 = 00100.
    \end{equation*}

    All in all, we exhausted all cases and found four bispecial factors:
    \begin{equation*}
        0010, 00100, 00100010, 001000100010.
    \end{equation*}
    This proves the claim.
\end{proof}

We give the extension graphs of these four bispecial factors in Figure~\ref{f:ext-001}. The longest among these, which has length 12, is the only one which is disconnected. We also saw that all the bispecial factors of $L$ of length at most 4 are connected, and it is not hard from there to complete the picture and show that $001000100010$ is both the \emph{longest bispecial factor starting with 001} and the \emph{smallest disconnected factor} of $L$. This can be done by explicit computations for the only three missing bispecial factors of length at most 12, which are $00010$, $000100$ and $000100010$.
\begin{figure}\centering
    \begin{minipage}[b]{0.23\linewidth}\centering
        \includegraphics[scale=.4]{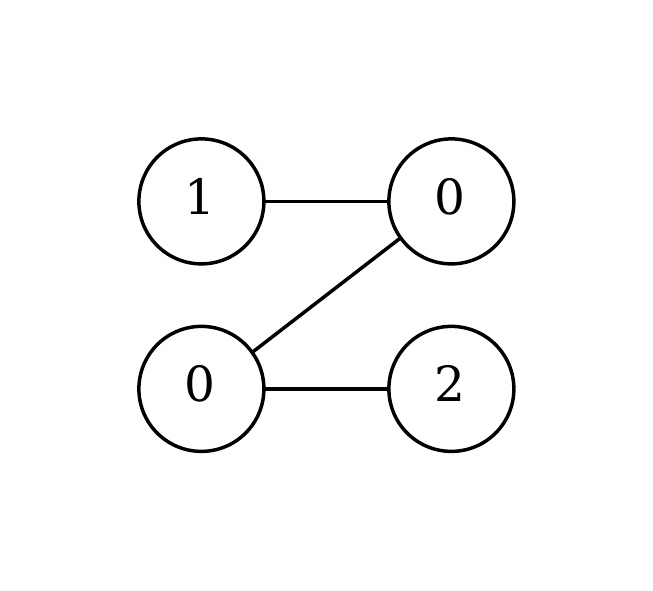}\\
        {\small$\ext(0010)$}
    \end{minipage}
    \begin{minipage}[b]{0.23\linewidth}\centering
        \includegraphics[scale=.4]{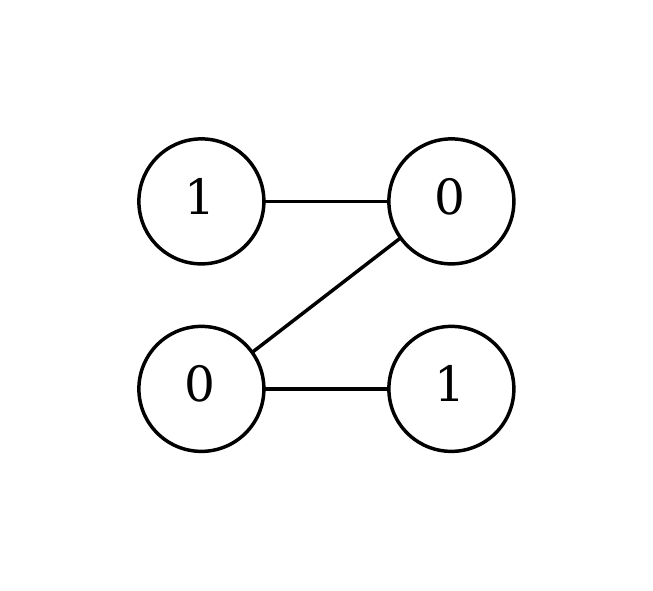}\\
        {\small$\ext(00100)$}
    \end{minipage}
    \begin{minipage}[b]{0.23\linewidth}\centering
        \includegraphics[scale=.4]{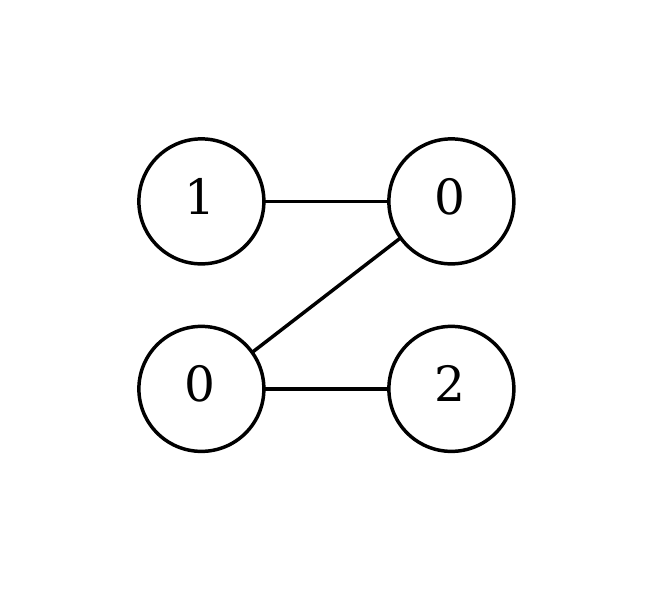}\\
        {\small$\ext(00100010)$}
    \end{minipage}
    \begin{minipage}[b]{0.23\linewidth}\centering
        \includegraphics[scale=.4]{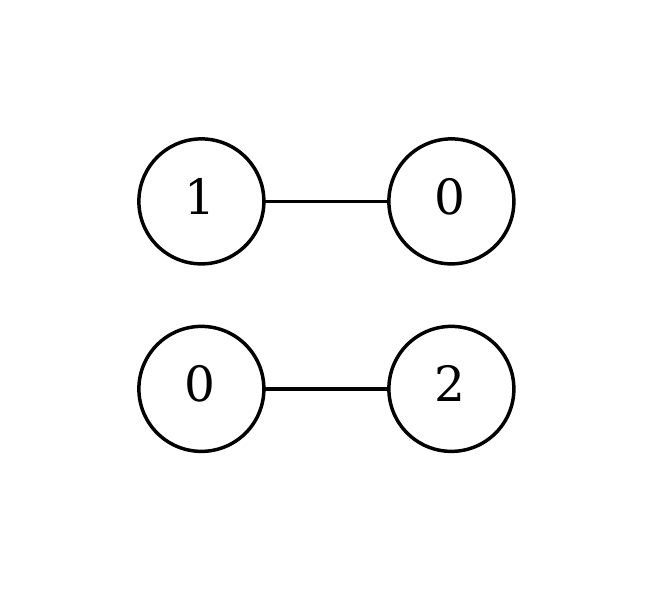}\\
        {\small$\ext(001000100010)$}
    \end{minipage}
    \caption{Extension graphs of all bispecial factors of $L$ starting with 001.}\label{f:ext-001}
\end{figure}

\subsection*{Step~\ref{step:stability}}
Next, we give conditions ensuring stability of some extension graphs under $x\mapsto \varphi(x)0$ or $x\mapsto \varphi(x)00$. 
\begin{claim}
    Let $x$ be a bispecial factor of $L$ starting with $000$. Then:
    \begin{equation*}
        \ext(x) \isom \begin{cases}
            \ext(\varphi(x)0) & \text{ if } x\in A^*00\\
            \ext(\varphi(x)00) & \text{ if } x\in A^*10.
        \end{cases}
    \end{equation*}
\end{claim}

\begin{proof}[Proof of the claim]
    Let us put:
    \begin{equation*}
        y = \begin{cases}
            \varphi(x)0 & \text{ if } x\in A^*00\\
            \varphi(x)00 & \text{ if } x\in A^*10.
        \end{cases}
    \end{equation*}
    Let $\sigma$ be the permutation of $A = \{ 0,1,2 \}$ fixing 0 and exchanging 1, 2. We readily deduce from Step~\ref{step:classification} that $\sigma(a)$ is a suffix of $\varphi(a)$, for all $a\in\lext(x)$. Similarly, if $x\in A^*00$, then $0\sigma(b)$ is a prefix of $\varphi(b)$, for all $b\in\rext(x)$; and if $x\in A^*10$, then $00\sigma(b)$ is a prefix of $\varphi(b)$, for all $b\in\rext(x)$. In particular, $\sigma(a)y\sigma(b)$ is a factor of $\varphi(axb)$. Thus,
    \begin{equation*}
        axb\in L \implies \varphi(axb)\in L \implies \sigma(a)y\sigma(b)\in L.
    \end{equation*}
    This shows that $\sigma\colon\ext(x)\to\ext(y)$ is a graph morphism, and that $y$ is bispecial. Moreover, it follows from Step~\ref{step:classification} that $\sigma$ is bijective on vertices. It remains only to show that $\sigma$ is onto on edges.

    Let us suppose that $\sigma(a)y\sigma(b)\in L$ for $a,b\in A$. We need to show that $axb\in L$. The fact that $\sigma(a)y\sigma(b)\in L$ implies:
    \begin{equation*}
        \exists z\in L, \ \varphi(z) = s\sigma(a)y\sigma(b)t.
    \end{equation*}
    Note that $y = \varphi(x)0$ or $\varphi(x)00$. In both cases, $\varphi(x)$ is a prefix of $y$ ending with 1 or 2. Moreover, we have $\sigma(a)\in\{1,2\}$ by Step~\ref{step:classification}. Thus, in the factorization $s\sigma(a)y\sigma(b)t$ given above, there must be one cutting point at the start of $y$, and one at the end of $\varphi(x)$. By Lemma~\ref{l:cutting}, this implies that $z$ has a factor of the form $cx'd$, where:
    \begin{enumerate}
        \item $\varphi(x') = \varphi(x)$;
        \item $\sigma(a)$ is a suffix of $\varphi(c)$;
        \item $\varphi(d)$ starts with $0\sigma(b)$ if $x\in A^*00$; or $00\sigma(b)$ if $x\in A^*10$.
    \end{enumerate}
    Injectivity of $\varphi$ implies $x = x'$, so $c\in\lext(x)$ and $d\in\rext(x)$. Using Step~\ref{step:classification}, we can then deduce from a case-by-case analysis that $a=c$ and $b=d$. This proves that $\sigma$ is onto on edges, thus finishing the proof.
\end{proof}

\subsection*{Step~\ref{step:wk}}
Recall that the longest bispecial factor of $L$ that starts with 001 is also its smallest disconnected element. We will see that \emph{all} disconnected elements of $L$ arise from this word. Consider the sequence of words $\{w_k\}_{k\in\nn}$ defined by:
\begin{equation*}
    w_0 = 001000100010, \qquad w_{k+1} =
    \begin{cases}
        \varphi(w_k)00 & \text{if $k$ is even};\\
        \varphi(w_k)0 & \text{if $k$ is odd}.
    \end{cases}
\end{equation*}

For the purpose of the proof below, it is useful to notice that $w_k$ ends with $10$ if $k$ is even, and with $00$ if $k$ is odd. We now prove the following claim.
\begin{claim}
   A word $w\in L$ is disconnected if and only if $w = w_k$ for some $k\in\nn$.  
\end{claim}

\begin{proof}
    We already saw that $w_0$ is disconnected, and one can check via explicit computations that so is $w_1$. Since $w_k$ starts with $000$ whenever $k\geq 1$, it follows from Step~\ref{step:stability} that $w_k$ is disconnected for all $k\in\nn$.

    For the converse, we proceed by induction on $|w|$. The smallest disconnected word, $w_0$, provides the basis for the induction. Let us consider a disconnected word $w\in L$ such that $|w|>|w_0|=12$. Since $w_0$ is also the longest bispecial factor starting with $001$ (see Step~\ref{step:001}), we may assume that $w$ starts with $000$. 

    We start by treating the case $w\in A^*00$. By Step~\ref{step:classification}, we know that:
    \begin{equation*}
        \lext(w) = \{1,2\},\quad \rext(w) = \{0,1\}.
    \end{equation*}
    Let us write $w = w'00$. Since $w0\in L$, it follows that $w'$ cannot end with $0$. Let us consider $z_1, z_2, z_3, z_4\in L$ such that:
    \begin{equation*}
        \varphi(z_1)=s_12wt_1,\quad \varphi(z_2)=s_21wt_2,\quad \varphi(z_3)=s_3w0t_3,\quad \varphi(z_4)=s_4w1t_4.
    \end{equation*}
    By repeatedly applying Lemma~\ref{l:cutting}, we deduce that: 
    \begin{itemize}
        \item $z_1$ has a factor of the form $1x_1$ such that $\varphi(x_1) = w'$.
        \item $z_2$ has a factor of the form $ax_2$ such that $\varphi(x_2) = w'$ and $a\in\{0, 2\}$.
        \item $z_3$ has a factor of the form $x_30$ such that $\varphi(x_3) = w'$.
        \item $z_4$ has a factor of the form $x_42$ such that $\varphi(x_4) = w'$.
    \end{itemize}
    Since $\varphi$ is injective, we deduce $x_1=x_2=x_3=x_4=x$, and $x$ is bispecial. Note that $|x|\leq 2$ would imply $|w| = |\varphi(x)|+2 \leq 10$, which is a contradiction. Thus, we may assume $|x|\geq 3$. Since $2$ is a right extension of $x$, we deduce by Step~\ref{step:classification} that $x\in A^*10$. By Step~\ref{step:wk}, $\ext(x) \isom \ext(\varphi(x)00) = \ext(w)$; thus, $x$ is disconnected. By induction, $x = w_k$ for some $k\in\nn$, and since $x$ ends with $10$, $k$ is even. Therefore $w = \varphi(w_k)00 = w_{k+1}$.

    The case $w\in A^*10$ is handled in a similar fashion. Let us go quickly over the argument. This time, we have
    \begin{equation*}
        \lext(w) = \{1,2\},\quad \rext(w) = \{0,2\}.
    \end{equation*}
    We write $w = w'0$. Take $z_1, z_2, z_3, z_4\in L$ such that:
    \begin{equation*}
        \varphi(z_1)=s_12wt_1,\quad \varphi(z_2)=s_21wt_2,\quad \varphi(z_3)=s_3w0t_3,\quad \varphi(z_4)=s_4w2t_4.
    \end{equation*}
    Again, it follows from Lemma~\ref{l:cutting} that: 
    \begin{itemize}
        \item $z_1$ has a factor of the form $1x_1$ such that $\varphi(x_1) = w'$.
        \item $z_2$ has a factor of the form $ax_2$ such that $\varphi(x_2) = w'$ and $a\in\{0,2\}$.
        \item $z_3$ has a factor of the form $x_3a$ such that $\varphi(x_3) = w'$ and $a\in\{0,2\}$.
        \item $z_4$ has a factor of the form $x_41$ such that $\varphi(x_4) = w'$.
    \end{itemize}
    By injectivity of $\varphi$, we have $x_1=x_2=x_3=x_4=x$ and $x$ is bispecial. Moreover, $|x|\leq 2$ would imply $|w| = |\varphi(x)|+1 \leq 9$, which contradicts our standing assumption that $|w|>12$. Thus, we may assume $|x|\geq 3$. Since $1$ is a right extension of $x$, it follows from Step~\ref{step:classification} that $x$ ends with $00$, so by Step~\ref{step:stability}, $\ext(x) \isom \ext(\varphi(x)0) = \ext(w)$. This implies that $x$ is disconnected, so by induction $x = w_k$ for some $k\in\nn$. As $x$ ends with $00$, $k$ is odd and $w = \varphi(w_k)0 = w_{k+1}$.
\end{proof}

This, combined with the graph isomorphism identified in Step~\ref{step:stability}, allows us to explicitly compute the extension graphs of all the disconnected words of $L$. These extension graphs are shown in Figure~\ref{f:ext-disconnected}.
\begin{figure}\centering
    \begin{minipage}[b]{.23\linewidth}\centering
        \includegraphics[scale=.4]{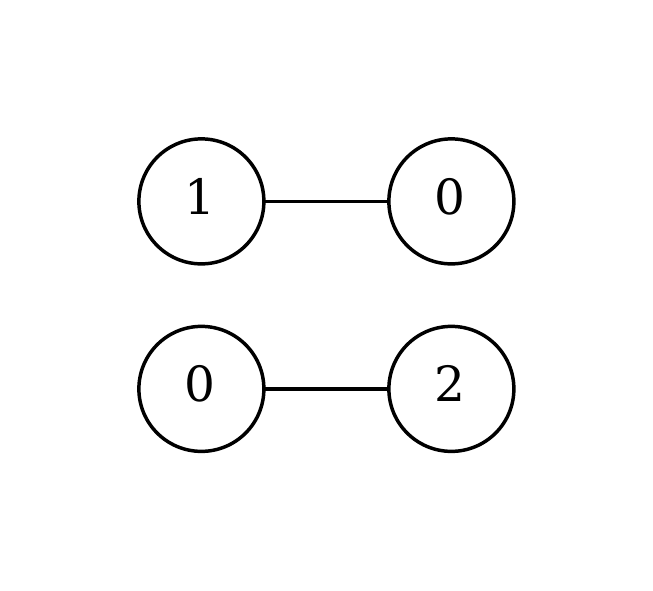}\\
        {\small$\ext(w_0)$}
    \end{minipage}
    \begin{minipage}[b]{.23\linewidth}\centering
        \includegraphics[scale=.4]{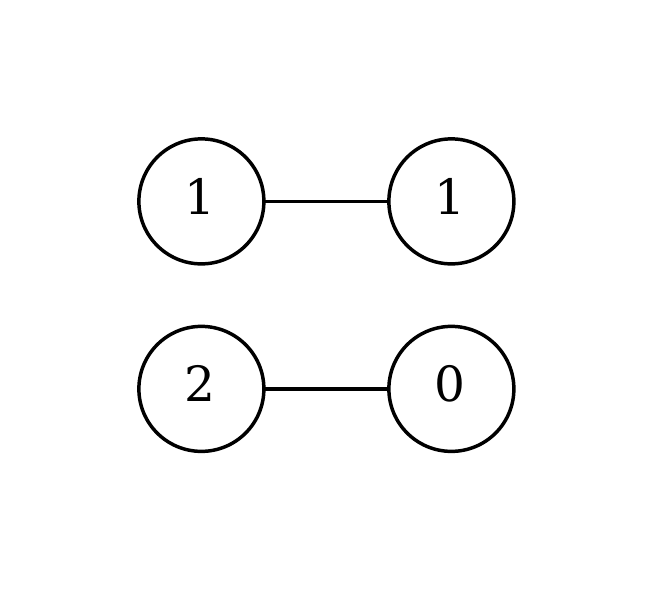}\\
        {\small$\ext(w_k)$, $k\geq 1$ odd}
    \end{minipage}
    \begin{minipage}[b]{.23\linewidth}\centering
        \includegraphics[scale=.4]{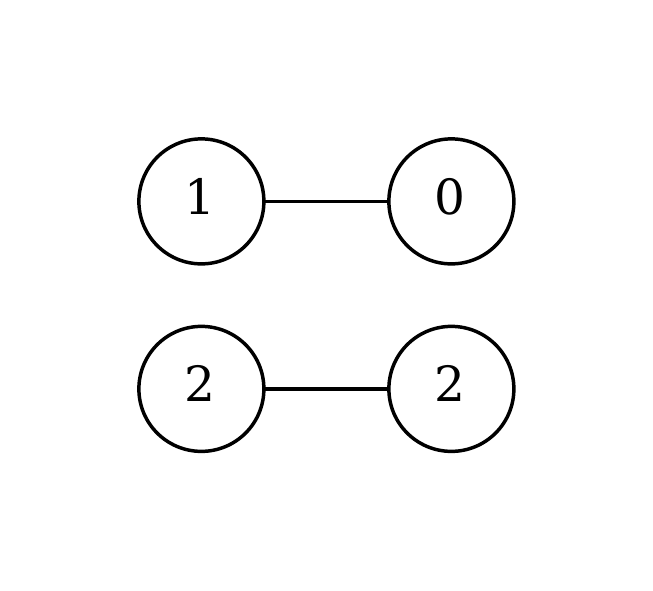}\\
        {\small$\ext(w_k)$, $k\geq 1$ even}
    \end{minipage}
    \caption{Extension graphs of the disconnected words of $L$.}\label{f:ext-disconnected}
\end{figure}

\subsection*{Step~\ref{step:conclude}}
Now that we know exactly which are the disconnected words of $L$, it remains to show that these words are suffix-connected. For $k\in\nn$, let us write $d_k = |\varphi^k(001)|+1$ and $y_k = \tail^{d_k-1}(w_k)$. This means that $w_k = \varphi^k(001)y_k$ and the depth $d_k$ suffix extension graph of $w_k$ is precisely $\ext_{d_k,d_k}(y_k)$. For the case $k=0$, we have $d_0=4$ and $y_0=000100010$. Notably, the depth~4 suffix extension graph of $w_0$, which is shown in Figure~\ref{f:suff-ext-w0}, is connected, hence $w_0$ is suffix-connected at depth~4. We will show that $w_k$ is suffix-connected at depth $d_k$ for all $k\in\nn$. 

Let us first note that the natural embedding of $\lext(w_k)$ in $\ext_{d_k,d_k}(y_k)$ is given by right multiplication by $\varphi^k(001)$. Before concluding the proof of Theorem~\ref{t:example}, we need to establish the following technical lemma, which gives some properties of the words $x_k = \init(\varphi^k(2))$.

\begin{figure}\centering
    \begin{minipage}[b]{.23\linewidth}\centering
        \includegraphics[scale=.4]{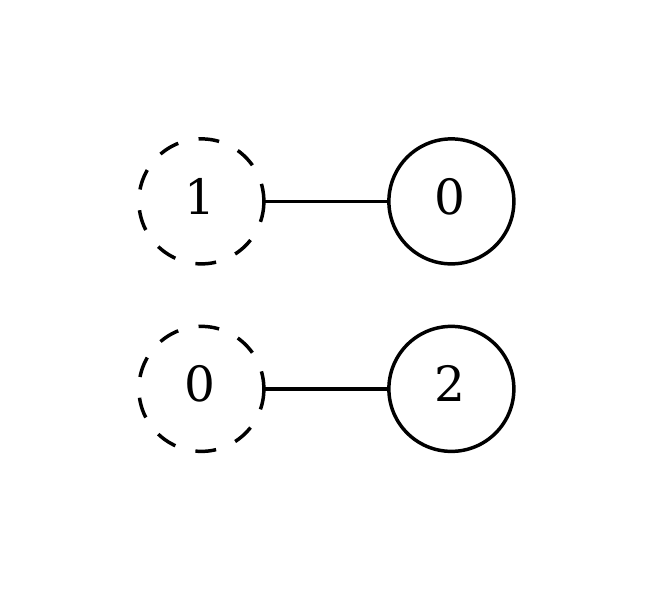}\\
        {\small$\ext_{1,1}(001000100010)$}
    \end{minipage}
    \begin{minipage}[b]{.23\linewidth}\centering
        \includegraphics[scale=.36]{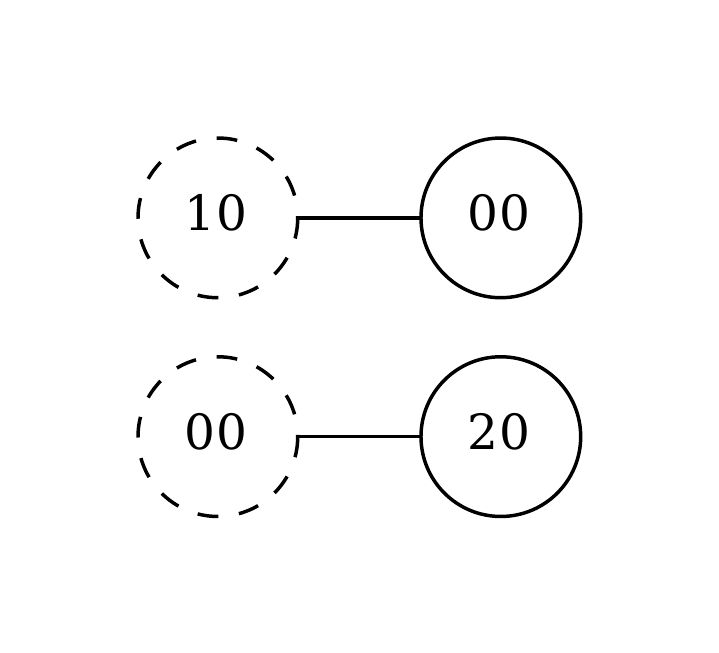}\\
        {\small$\ext_{2,2}(01000100010)$}
    \end{minipage}
    \begin{minipage}[b]{.23\linewidth}\centering
        \includegraphics[scale=.32]{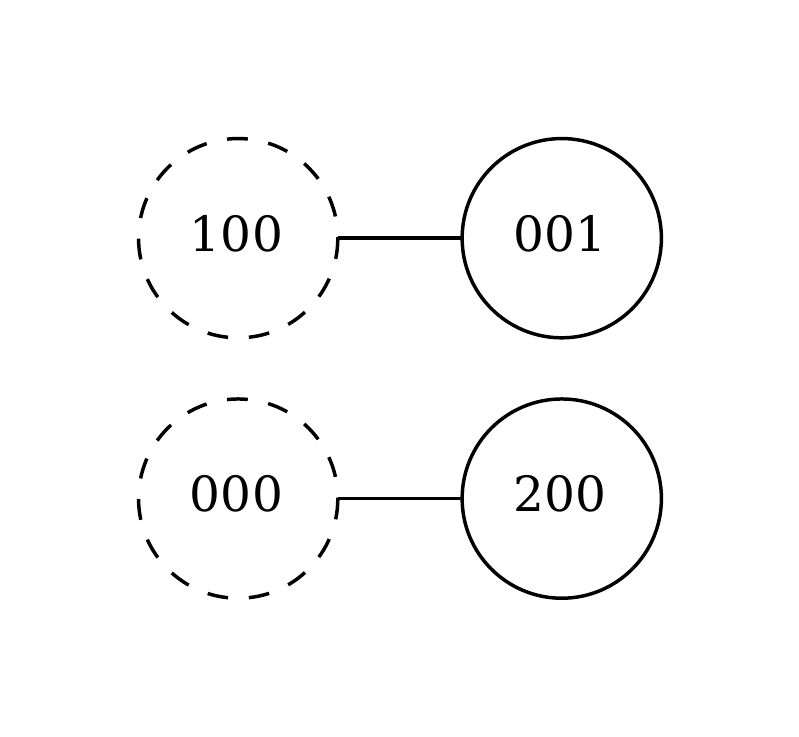}\\
        {\small$\ext_{3,3}(1000100010)$}
    \end{minipage}
    \begin{minipage}[b]{.23\linewidth}\centering
        \includegraphics[scale=.3]{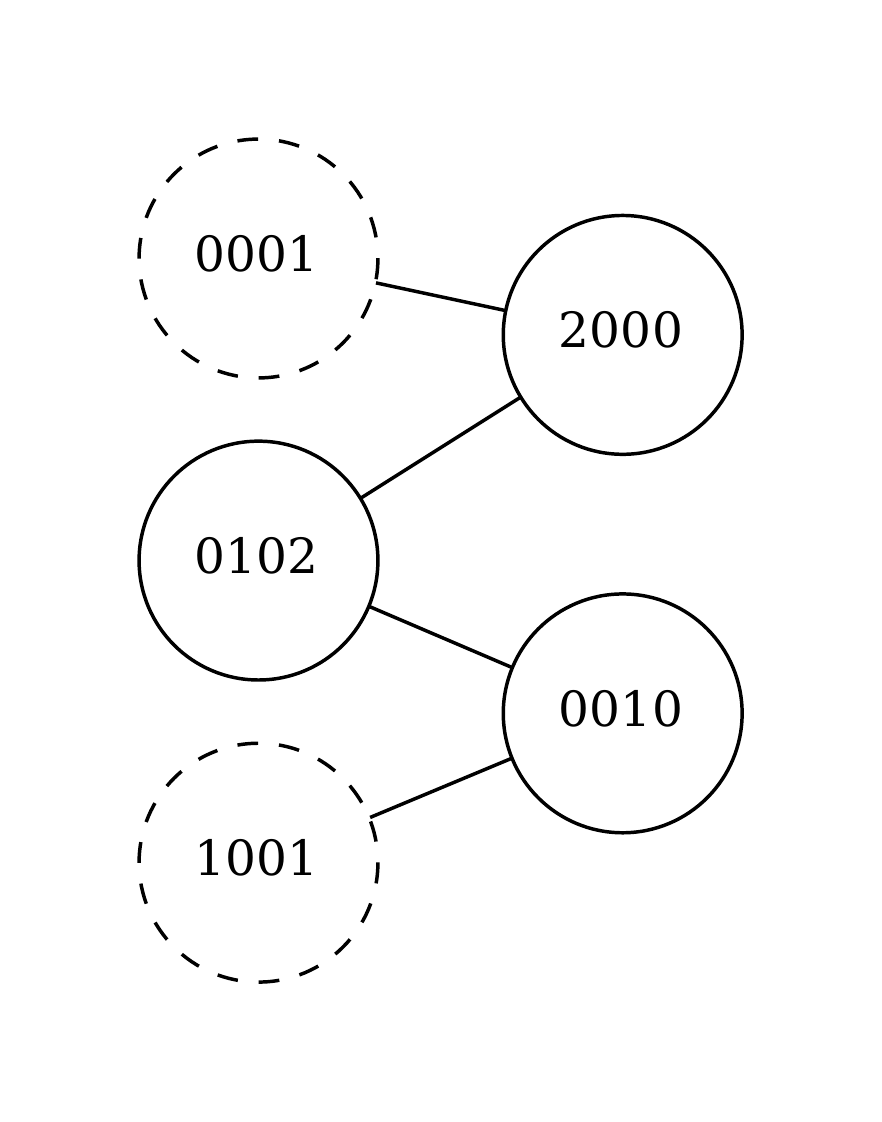}\\
        {\small$\ext_{4,4}(000100010)$}
    \end{minipage}
    \caption{Suffix extension graphs of $w_0=001000100010$ at depth up to 4. The dashed vertices represent the natural embeddings of $\lext(w_0)$.}\label{f:suff-ext-w0}
\end{figure}

\begin{lemma}
    For all $k\in\nn$, the following hold:
    \begin{enumerate}
        \item $x_{k+1} = \begin{cases} \varphi(x_k)00 & \text{ if $k$ is even;}\\ \varphi(x_k)0 & \text{ if $k$ is odd.} \end{cases}$
        \item $x_k0$ is a prefix of $\varphi^k(0)$.
    \end{enumerate}
\end{lemma}

\begin{proof}
    \textbf{(1)} If $k$ is even, then $\varphi^k(2) = x_k2$ and $\varphi^{k+1}(2) = x_{k+1}1$. It follows that
    \begin{equation*}
        x_{k+1}1 = \varphi(x_k2) = \varphi(x_k)001.
    \end{equation*}
    Hence, the result follows. Similarly, if $k$ is odd, $\varphi^k(2) = x_k1$, $\varphi^{k+1}(2) = x_{k+1}2$, and
    \begin{equation*}
        x_{k+1}2 = \varphi(x_k1) = \varphi(x_k)02.
    \end{equation*}

    \textbf{(2)} We proceed by induction on $k$. The basis, $k=0$, is obvious. Let us assume $\varphi^k(0) = x_k0t_k$, for some $t_k\in A^*$. Hence,
    \begin{equation*}
        \varphi^{k+1}(0) = \varphi(x_k0t_k) = \varphi(x_k)0001\varphi(t_k) = 
            \begin{cases} 
                x_{k+1}01\varphi(t_k) & \text{ if $k$ even;}\\
                x_{k+1}001\varphi(t_k)& \text{ if $k$ odd.}
            \end{cases}\qedhere
        \end{equation*}
\end{proof}

By the recursive definition of $w_k$ and a straightforward inductive argument involving Part (1) of the lemma, we have $\varphi^k(y_0)x_k=y_k$ for all $k\in\nn$. Moreover, note that $x_k$ is a prefix of both $\varphi^k(2000)$ and $\varphi^k(0010)$, the former by definition and the latter by part (2) of the lemma. Since $2000$ and $0010$ are right extensions of $y_0$ (see Figure~\ref{f:suff-ext-w0}), it follows that $x_k^{-1}\varphi^k(2000)$ and $x_k^{-1}\varphi^k(0010)$ are right extensions of $y_k$. With these observations in mind, we are ready to conclude the proof of Theorem~\ref{t:example}. We do this by establishing the following claim.
\begin{claim}
    For $k\geq 1$, there is a path in $\ext_{d_k,d_k}(y_k)$ between $1\varphi^k(001)$ and $2\varphi^k(001)$.
\end{claim}

\begin{proof}[Proof of the claim]
    Consider the map $\sigma_k\colon\ext_{d_0,d_0}(y_0)\to\ext_{d_k,d_k}(y_k)$ defined as follows: an element $u\in\lext_{d_0}(y_0)=\{0001,0102,1001\}$ is mapped to the suffix of length $d_k$ of $\varphi^k(u)$, and an element $v\in\rext_{d_0}(y_0) = \{2000, 0010\}$ is mapped to the prefix of length $d_k$ of $x_k^{-1}\varphi^k(v)$. We first need to show that this map is well-defined. This amounts to show that $|\varphi^k(u)|\geq d_k$ for all $u\in\lext(y_0)$, and $|\varphi^k(v)| - |x_k|\geq d_k$ for all $v\in\rext(y_0)$. The former condition is obvious, and the latter boils down to a few computations:
    \begin{align*}
        |\varphi^k(0010)| - |x_k| & = |\varphi^k(001)| + |\varphi^k(0)| - |\varphi^k(2)| + 1 \\
                                  & > |\varphi^k(001)| + 1 = d_k;\\
        |\varphi^k(2000)| - |x_k| & = |\varphi^k(2000)| - |\varphi^k(2)| + 1 \\
                                  & = |\varphi^k(000)| + 1 \\
                                  & > |\varphi^k(001)|+1 = d_k.
    \end{align*}

    Note that $\sigma_k$ maps $\{0001, 1001\}$ onto the natural embedding of $\lext(w_k)$. Since $\ext_{d_0,d_0}(y_0)$ is connected, it suffices to show that $\sigma_k$ defines a graph morphism. Take $u\in\lext_{d_0}(y_0)$ and $v\in\rext_{d_0}(y_0)$, and suppose that $uy_0v\in L$. Then, it follows that $\varphi^k(uy_0v)\in L$. Since $\sigma_k(u)$ is a suffix of $u$ and $x_k\sigma_k(v)$ is a prefix of $v$, we conclude that $\sigma_k(u)y_k\sigma_k(v) = \sigma_k(u)\varphi^k(y_0)x_k\sigma_k(v)$ is a factor of $\varphi^k(uy_0v)$. Therefore, it must also be in $L$, and $\sigma_k\colon\ext_{d_0,d_0}(y_0)\to\ext_{d_k,d_k}(y_k)$ is a graph morphism.
\end{proof}

With some extra work, we were able to show that the map $\sigma_k$ defined in the previous proof is in fact a graph isomorphism. To prove this, we made use of the following observation, which is a consequence of the Cayley-Hamilton theorem: for any word $x\in A^*$, the sequence $(|\varphi^k(x)|)_{k\in\nn}$ follows the linear recurrence determined by the characteristic polynomial of $\varphi$. This is a general observation which holds for any substitution, and we believe it could be useful for establishing suffix-connectedness in harder cases.

\section{Conclusion}
\label{s:conclu}

Let us end this paper by suggesting a few ideas for future reasearch. 

Firstly, we feel that the proof presented in Section~\ref{s:example}, on account of its ad-hoc and technical nature, is somewhat unsatisfactory. We hope it could be improved.
\begin{question}
    Is there a more systematic approach to show that a given language is suffix-connected?
\end{question}

In particular, it could be interesting to study other examples of suffix-connected languages defined by primitive substitutions, and see how much of Section~\ref{s:example} can be recycled. According to our computations, the languages defined by the following primitive substitutions are likely to be suffix-connected while also having infinitely many disconnected elements:
\begin{equation*}
    \begin{array}{lll}
        0 & \mapsto & 100 \\
        1 & \mapsto & 032 \\
        2 & \mapsto & 232 \\
        3 & \mapsto & 03
    \end{array};
        \qquad
    \begin{array}{lll}
        0 & \mapsto & 01 \\
        1 & \mapsto & 2 \\
        2 & \mapsto & 3 \\
        3 & \mapsto & 02
    \end{array};
        \qquad
    \begin{array}{lll}
        0 & \mapsto & 12 \\
        1 & \mapsto & 2 \\
        2 & \mapsto & 01 \\
          & & 
    \end{array}.
\end{equation*}

In \cite{Dolce2020}, Dolce and Perrin introduced the notion of an \emph{eventually dendric language}, which requires all but finitely many extension graphs to be trees. This suggests the analogous notion of an \emph{eventually suffix-connected language}, in which all but finitely many words are suffix-connected. 

\begin{question}
    Can we find a generalization of Theorem~\ref{t:main} for eventually suffix-connected languages?
\end{question}

Finally, Dolce and Perrin also showed that the class of eventually dendric languages is closed under two operations, namely conjugacy and complete bifix decoding \cite{Dolce2020}. We wonder if analogous results hold for suffix-connected languages.
\begin{question}
    Is the class of suffix-connected languages closed under complete bifix decoding or conjugacy?
\end{question}

\providecommand{\bysame}{\leavevmode\hbox to3em{\hrulefill}\thinspace}

\end{document}